\documentclass[preprint,12pt]{elsarticle}

\usepackage{amssymb}
\usepackage{amsthm}

\usepackage{mathtools}
\usepackage[utf8]{inputenc}
\usepackage{bbm, bm}
 \usepackage{caption}
 \usepackage{algorithm}
\usepackage{algorithmic}
\usepackage{graphicx}
\usepackage{varwidth}
\usepackage[dvipsnames]{xcolor}

\DeclareMathOperator*{\esssup}{ess\,sup}

\newcommand{\R}{\mathbb{R}}
\newcommand{\N}{\mathbb{N}}

\newcommand{\E}{\mathbb{E}}

\newcommand{\D}{\,\mathrm{d}}

\newcommand{\cA}{\mathcal{A}}
\newcommand{\cN}{\mathcal{N}}
\newcommand{\cB}{\mathcal{B}}
\newcommand{\cb}{\mathfrak{b}}
\newcommand{\cS}{\mathcal{S}}
\newcommand{\cJ}{\mathcal{J}}
\newcommand{\cG}{\mathcal{G}}
\numberwithin{equation}{section}

\newtheorem{theorem}{Theorem}[section]
\newtheorem{lemma}[theorem]{Lemma}

\newtheorem{proposition}[theorem]{Proposition}

\theoremstyle{remark}
\newtheorem{remark}[theorem]{Remark}

\theoremstyle{definition}
\newtheorem{example}[theorem]{Example}

\newcounter{assumption}
\newtheorem{assumption}[theorem]{Assumption}
\newcounter{subassumption}[assumption]
\renewcommand{\thesubassumption}{(\textit{\roman{subassumption}})}
\makeatletter
\renewcommand{\p@subassumption}{\theassumption}
\makeatother
\newcommand{\subasu}{
  \refstepcounter{subassumption}%
  \thesubassumption~\ignorespaces}

\usepackage{algorithm}
\usepackage{algorithmic}

\usepackage{todonotes}

\begin{document}

\begin{frontmatter}



\title{A stochastic gradient method for a class of nonlinear PDE-constrained optimal control problems under uncertainty}


\author{Caroline Geiersbach\footnote{Weierstrass Institute, Berlin (Germany).  caroline.geiersbach@wias-berlin.de }}

\author{Teresa Scarinci\footnote{Dipartimento di Ingegneria Elettrica e dell'Informazione ``Maurizio Scarano", Università degli studi di Cassino e del Lazio Meridionale (Italy). teresa.scarinci@unicas.it}
}



\begin{abstract}
The study of optimal control problems under uncertainty plays an important role in scientific numerical simulations. This class of optimization problems is strongly utilized in engineering, biology and finance. In this paper, a stochastic gradient method is proposed for the numerical resolution of a nonconvex stochastic optimization problem on a Hilbert space. We show that, under suitable assumptions, strong or weak accumulation points of the iterates produced by the method converge almost surely to stationary points of the original optimization problem. Measurability and convergence rates of a stationarity measure are handled, filling a gap for applications to nonconvex infinite dimensional stochastic optimization problems. The method is demonstrated on an optimal control problem constrained by a class of elliptic semilinear partial differential equations (PDEs) under uncertainty. 
\end{abstract}




\end{frontmatter}

\section{Introduction}\label{section:introduction}
Uncertainty appears in many optimization problems, for instance in PDE-constrained optimization, where the data used in the PDE may be unknown or based on measurements. Optimal solutions may exhibit a significant dependence on those data and mathematical formulations of the optimization problem should take into account these data, or uncertainties. This class of optimization problems is strongly utilized in engineering, biology, finance, and economics (see, e.g., \cite{GL2022, GKV2019, Kouri2018, Papa1995} and references therein). A popular modeling choice is to optimize over the average of the objective function that is parametrized with respect to the uncertainty. The corresponding stochastic optimization problem is formally written as
\begin{equation}\label{eq:main_prb}
\min_{u\in H} \, \left \lbrace j(u)= \mathbb{E}[\mathcal{J}(u,\xi)] = \int_\Omega \mathcal{J}(u,\xi(\omega)) \D P(\omega) \right\rbrace.
\end{equation}
This type of problem is sometimes called risk-neutral to distinguish it from risk-averse (cf.~\cite[Chapter 6]{Shapiro}) formulations. A nice introduction to applications and modeling choices for optimal control problems involving PDEs can be found in \cite{frutos}.

In this paper, we focus on a nonconvex stochastic optimization problem of the form \eqref{eq:main_prb}. In view of applications to PDE-constrained optimization, we assume that the control space $H$ is a real Hilbert space. Here, $\xi\colon \Omega \rightarrow \Xi$ is a (measurable) vector-valued random variable defined on a given probability space $(\Omega,\mathcal{F}, P)$ with images in a (real) complete separable metric space $\Xi$. The expectation is denoted by $\mathbb{E}[\cdot]$ and a function $\mathcal{J}\colon H \times \Xi \rightarrow \mathbb{R}$ is given.

Our focus is on the application of the classical stochastic gradient method for finding a stationary point to Problem \eqref{eq:main_prb} with an emphasis on investigating its properties in the case where $j$ is smooth but possibly nonconvex. The method is described in its most general form in Algorithm~\ref{alg:PSG_Hilbert_Nonconvex_Unconstrained}. 
\begin{algorithm}[H] 
\begin{algorithmic}[0] 
\STATE \textbf{Initialization:} Take an initial $u_1 \in H$.
\FOR{$n=1,2,\dots$}
\STATE Generate $\xi_n \in \Xi$, independent of $\xi_1, \dots, \xi_{n-1}$.
\STATE Set $u_{n+1} := u_n - t_n \mathcal{G}(u_n,\xi_n)$, $t_n \geq 0$.
\ENDFOR
\end{algorithmic}
\captionof{algorithm}{Stochastic Gradient Method} 
\label{alg:PSG_Hilbert_Nonconvex_Unconstrained}
\end{algorithm}
Algorithm~\ref{alg:PSG_Hilbert_Nonconvex_Unconstrained} can be seen as a stochastic approximation of the gradient descent method, since it replaces the actual gradient $\nabla j$ by an estimate calculated from a randomly selected realization from the sample space. It relies on the existence of a \emph{stochastic gradient} $\mathcal{G}\colon H\times \Xi \rightarrow H$ such that $\E[\mathcal{G}(u,\xi)] \approx \nabla j(u).$ In Algorithm \ref{alg:PSG_Hilbert_Nonconvex_Unconstrained} the iterate $u_{n+1}$ depends on the stochastic gradient $\mathcal{G}$ evaluated along the sequence $(u_n,\xi_n)$. The price that is paid for the noisy gradient is the step-size $t_n$, which needs to decrease at the appropriate rate to ensure convergence, which can cause the method to stall. The appropriate choice of the step-size, given later in \eqref{eq:Robbins-Monro-stepsizes}, appeared in the classical root-finding method by Robbins and Monro from 1951 \cite{RM}. In spite of these practical difficulties, the stochastic gradient method is now a well-known tool in several fields of applied mathematics, such as machine learning and optimization (see, e.g., \cite{BZ, Bottou, duchi, MKN} and the references therein).  

The convergence of the (projected) stochastic gradient method on a Hilbert space was shown for convex problems in \cite{CC} and for convex problems with bias in \cite{GP}. Recently, we showed convergence of the stochastic proximal gradient method for nonconvex problems on a Hilbert space in \cite{GS2021}. This paper aims to complement the results in \cite{GS2021} by focusing on the method's properties when applied to problems on a Hilbert space that are smooth but possibly nonconvex. While the focus in \cite{GS2021} was in obtaining asymptotic convergence results, here we discuss other aspects of the simpler method, including convergence rates and measurability.



To motivate the application and to fix ideas, we
follow \cite{KS} and consider a stochastic optimal control problem constrained by the following boundary problem for a second-order elliptic semilinear PDE under uncertainty: 
\begin{equation}\label{eq:PDEs_intro}
\begin{aligned}
-\nabla \cdot (k(x,\omega)\nabla y(x,\omega))+c(x,\omega)y(x,\omega)+N(y(x,\omega),x,\omega)                        &= [B(\omega)u](x)+b(x,\omega),  \\
k(x,\omega)\frac{\partial y}{\partial n}(x,\omega)&=0,
\end{aligned}
\end{equation}
to be satisfied for a.e.~$\omega \in \Omega$ and $x$ in a bounded domain $D \subset \R^d$. 
The random fields $k\colon D \times \Omega \rightarrow \mathbb{R}$, $c\colon D \times \Omega \rightarrow \mathbb{R}$ and $b\colon D \times \Omega \rightarrow \mathbb{R}$ are given and $N\colon H^1(D)\times D \times \Omega \rightarrow \mathbb{R}$ is a random nonlinear operator. The function $u\in L^2(D)$ plays the role of the control and $B(\cdot)$ is an (linear) operator-valued random variable. Existence and uniqueness results for \eqref{eq:PDEs_intro} are classical for a fixed realization $\omega$; regularity of the solution with respect to the Bochner space $L^q(\Omega, H^1(D))$ for appropriate $q \geq 1$ can also be shown; see \cite{KS}.

Now, for a given $\tilde{J}\colon H^1(D)\rightarrow \mathbb{R}$  and $\rho\colon L^2(D)\rightarrow \mathbb{R}$, consider the following optimal control problem
\begin{equation}\label{eq:main_prb_PDEs}
\min_{u\in L^2(D)}\lbrace \mathbb{E}[\tilde{J}(y)] +\rho(u) \rbrace \quad
\mbox{subject to }\quad y=y_u(\cdot,\omega) \mbox{ solves } \eqref{eq:PDEs_intro}.
\end{equation}
An example of a cost functional often encountered in such applications is a tracking-type functional with Tikhonov regularization, i.e., 
\begin{equation}\label{eq:example_problem}
\tilde{J}(y):= \frac{1}{2}\| y-y_D \|^2_{L^2(D)},\; \quad \rho(u) :=\frac{\lambda}{2}\| u \|_{L^2(D)}^2,
\end{equation}
for a given $\lambda\geq 0$ and a desired target $y_D\in L^2(D)$. 

In simulations, the random functions are generated using random vectors; i.e., $k(x,\omega)=\hat{k}(x,\xi^k(\omega))$ for appropriate $\hat{k}$, and likewise for the other elements. An example of this structure is a truncated Karhunen--Lo\`eve expansion; see the example in Section~\ref{sec:numerics}.
Collecting the random elements, we set $\xi=(\xi^k, \xi^c, \xi^b, \xi^N, \xi^B).$  
If a solution to \eqref{eq:PDEs_intro} exists, the control-to-state operator $\mathcal{S}$ and corresponding parametrized $S$ defined by 
\begin{equation}\label{eq:param_S}
\cS(\xi(\omega))u=S(\omega)u =y_u(\omega) \in H^1(D)
\end{equation}
are well-defined \footnote{We underline the following fact: by the Doob-Dynkin's Lemma, 
if the random fields are generated by a random finite-dimensional vector $\xi\colon \Omega \rightarrow \Xi \subset \R^m$ with uncorrelated random variables (i.e., finite-dimensional noise), then the solution $y$ of \eqref{eq:PDEs_intro} is also finite-dimensional noise and $y=\hat y (x, \xi)$; see \cite[Section 4.1]{frutos}. 
} and it is possible to reduce Problem \eqref{eq:main_prb_PDEs} to a problem of the form \eqref{eq:main_prb} with $H=L^2(D).$ Indeed, we can define the reduced (random variable) objective by
\begin{equation*}
\label{eq:reduced-functional}
\mathcal{J}(u,\xi):= \tilde{J}({\cS}(\xi)u)+\rho(u).
\end{equation*}
Let us remark that since the PDE constraint \eqref{eq:PDEs_intro} is semilinear, the mapping $u\mapsto S(\omega)u$ is nonlinear and consequently $j$ cannot expected to be convex.

The paper is organized as follows. After preliminaries, Section \ref{sub:stoch_grad_method} includes the following contributions:
\begin{itemize}
\item Lemma~ \ref{lem:sufficient-conditions-boundedness}: Sufficient conditions for ensuring almost sure boundedness of the iterates generated by Algorithm~\ref{alg:PSG_Hilbert_Nonconvex_Unconstrained}.
    \item Lemma~\ref{lem:measurability}: Sufficient conditions ensuring the measurability of the sequence $\lbrace u_n \rbrace$ of vector-valued iterates generated by Algorithm~\ref{alg:PSG_Hilbert_Nonconvex_Unconstrained}, filling a theoretical gap.
    \item Theorem~\ref{theorem:convergence-algorithm1}: Almost sure convergence of weak or strong accumulation points of $\lbrace u_n \rbrace$ to stationary points of Problem~\eqref{eq:main_prb}.
    \item Example~\ref{example-line-search}: An example showing how a naive application of the Armijo line search procedure  causes the stochastic gradient method to diverge, further demonstrating the necessity of the Robbins--Monro step-size rule.
    \item Theorem~\ref{thm:convergence-rate}: Almost sure convergence rates of the square norm of the gradient.
\end{itemize}
Section \ref{sec:OCforPDEs} introduces a class of semilinear elliptic PDEs that is more general than \eqref{eq:main_prb_PDEs}. The stochastic gradient is defined and for the optimal control problem, the conditions from Section \ref{sub:stoch_grad_method} ensuring convergence are verified. 
Section~\ref{sec:numerics} contains numerical experiments demonstrating convergence rates from Theorem~\ref{thm:convergence-rate}.

\section{Convergence results}\label{sub:stoch_grad_method}
\subsection{Preliminaries}\label{sec:preliminaries}
Throughout, $(\Omega,\mathcal{F},P)$ will denote a complete probability space, where $\Omega$ represents
the sample space, $\mathcal{F}\subset 2^{\Omega}$ is the $\sigma$-algebra of events on the power set of 
$\Omega$, denoted by $2^{\Omega}$, and $P\colon\Omega \rightarrow [0,1]$ is a probability measure.  
The operator $\mathbb{E}\left[ \cdot \right] $ denotes the expectation with respect to this distribution; for a random variable $\xi\colon \Omega \rightarrow \mathbb{R}$, this is defined by
\begin{equation}
\label{eq:expectation-definition-rv}
\mathbb{E}[\xi]=\int_{\Omega} \xi(\omega)\; \D P(\omega).
\end{equation}
For a Banach space $(X, \lVert \cdot \rVert_X)$, we denote the dual space by $(X^*, \lVert \cdot \rVert_{X^*})$ and the dual pairing by $\langle \cdot, \cdot \rangle_{X^*,X}.$ 
An $X$-valued random variable $y$ is Bochner integrable if there exists a
sequence $\{ y_n\}$ of $P$-simple functions $y_n\colon \Omega \rightarrow X$
such that $\lim_{n \rightarrow \infty} \int_{\Omega} \lVert y_n(\omega)-y(\omega) \rVert_X \D P(\omega) = 0$.
The limit of the integrals of $y_n$ gives the Bochner integral
(the expectation), defined analogously to \eqref{eq:expectation-definition-rv} by
\begin{equation*}
  \E[y]=\int_\Omega y(\omega) \D P(\omega) = \lim_{n \rightarrow \infty} \int_{\Omega} y_n(\omega) \D P(\omega).
\end{equation*}
Clearly, this expectation is an element of $X$. The Bochner space $L^p(\Omega,X)$ is the set of all (equivalence classes of) strongly measurable functions $y\colon\Omega \rightarrow X$ having finite norm, where the norm is defined by
$$\lVert y \rVert_{L^p(\Omega,X)}:= \begin{cases}
                                     (\int_\Omega \lVert y(\omega) \rVert_X^p \D P(\omega))^{1/p}, \quad &p < \infty\\
                                     \esssup_{\omega \in \Omega} \lVert y(\omega) \rVert_X, \quad &p=\infty
                                    \end{cases}.
$$
Let $\mathcal{L}(Y,Z)$ denote the space of all bounded linear operators from $Y$ to $Z$ (both Banach spaces). We recall that a mapping  $\cA\colon \Omega \rightarrow \mathcal{L}(Y,Z)$ is said to be a uniformly measurable operator-valued function (alternatively called a uniform random operator) if there exists a sequence of countably-valued operator-valued random variables $\{\cA_n\}$ in $\mathcal{L}(Y,Z)$ converging almost surely to $\cA$ in the uniform operator topology.
The set $L^p(\Omega,\mathcal{L}(Y,Z))$ corresponds to all uniform random operators
$\cA\colon \Omega \rightarrow \mathcal{L}(Y,Z)$ such that $(\int_{\Omega} \lVert \cA(\omega)\rVert_{\mathcal{L}(Y,Z)}^p \D P(\omega))^{1/p} < \infty$ if $p \in [1,\infty)$ and $\esssup_{\omega \in \Omega} \lVert \cA(\omega)\rVert_{\mathcal{L}(Y,Z)} < \infty$ for the case $p=\infty$.
The adjoint and inverse operators of random operators are to be understood in the almost sure sense; e.g., for $\cA$, the adjoint operator is the uniformly measurable operator-valued function $\cA^*$ such that for all $(y, z^*) \in Y \times Z^*$, 
\begin{align*}
P(\{\omega \in \Omega: \langle z^*, \cA(\omega)y \rangle_{Z^*,Z} =  \langle  \cA^{*}(\omega)z^*,y \rangle_{Y^*,Y} \}) = 1.
\end{align*} 
Recall that if $Y$ is separable, strong and weak measurability of the mappings of the form $y\colon \Omega \rightarrow Y$ coincide (cf.~\cite[Corollary 2, p.~73]{Hille1996}).

A filtration is a sequence $\{ \mathcal{F}_n\}$ of sub-$\sigma$-algebras of $\mathcal{F}$ such that {$\mathcal{F}_1 \subset \mathcal{F}_2 \subset \cdots \subset \mathcal{F}.$} 
Given a Banach space $X$, we define a discrete $X$-valued stochastic process as a collection of $X$-valued random variables indexed by $n$, in other words, the set $\{ \beta_n\colon \Omega \rightarrow X : n \in \N\}.$ 
The stochastic process is said to be adapted to a filtration $\{ \mathcal{F}_n \}$ if and only if $\beta_n$ is $\mathcal{F}_n$-measurable for all $n$. Suppose $\mathcal{B}(X)$ denotes the set of Borel sets of $X$.
The natural filtration is the filtration generated by the sequence $\{\beta_n\}$ and is given by $\mathcal{F}_n = \sigma(\{\beta_1, \dots ,\beta_n\}) = \{ \beta_i^{-1}(B)\colon B \in \mathcal{B}(X), i = 1, \dots, n\}$.
If for an event $F \in \mathcal{F}$ it holds that $P(F) = 1$, or equivalently, $P({\Omega}\backslash F) = 0$, we say $F$ occurs almost surely (a.s.); we equivalently use the shorthand a.e.~$\omega \in F$. Sometimes we also say that such an event occurs with probability one. 
A sequence of random variables $\{\beta_n\}$ is said to converge almost surely to a random variable $\beta$ if and only if 
$P\left(\left\lbrace \omega \in \Omega: \lim_{n \rightarrow \infty} \beta_n(\omega) = \beta(\omega) \right\rbrace\right) = 1.$
For an integrable random variable $\beta:\Omega \rightarrow \R$, the conditional expectation is denoted by $\E[\beta | \mathcal{F}_n]$, which is itself a random variable that is $\mathcal{F}_n$-measurable and which satisfies $\int_A \E[\beta | \mathcal{F}_n](\omega) \D P(\omega) = \int_A \beta(\omega) \D P(\omega)$ for all $A \in \mathcal{F}_n$. Almost sure convergence of $X$-valued stochastic processes and conditional expectation are defined analogously.

For an open subset $V$ of a Hilbert space $H$ and a function $j \colon V \rightarrow \R$, 
the Fr\'echet derivative of $j$ at $u\in V$ is denoted  by $j'(u)\in \mathcal{L}(H,\R)=H^*$. The inner product is denoted by $(\cdot,\cdot)_H$. 
For a Fr\'echet differentiable function $j\colon V \rightarrow \R$, the gradient $\nabla j \colon V \rightarrow H$ is the Riesz representation of the map $j':V \rightarrow H^*$, i.e.,~it satisfies $( \nabla j(u), v )_H = \langle j'(u),v \rangle_{H^*,H}$ for all $u \in H$ and $v \in H.$  The set $C_L^{1,1}(V)$ contains all continuously differentiable functions on $V \subset H$ with an $L$-Lipschitz gradient, meaning
$\lVert \nabla  j(u) - \nabla j(v) \rVert_H \leq L \lVert u - v \rVert_H$
is satisfied for all $u,v \in V.$ Suppose now $j \in C_L^{1,1}(V)$, $V\subset H$ open and convex. Then for all $u, v \in H$,
\begin{align}\label{eq:taylor_est}
j(v) +  (\nabla j(v), u-v)_H &- \frac{L}{2} \lVert u - v \rVert^2_H \leq  j(u) \leq \\ 
&j(v) + (\nabla j(v), u-v)_H + \frac{L}{2} \lVert u-v\rVert^{2}_H.\nonumber
\end{align}
The second Fr\'echet derivative at $u$ is denoted by $j''(u) \in \mathcal{L}(H, H^*)$; the Hessian is denoted by $\nabla^2 j(u) \in \mathcal{L}(H,H)$.
Weak and strong convergence is denoted by $\rightharpoonup$ and $\rightarrow$, respectively.
An operator $T\colon H_1 \rightarrow H_2$ between two Hilbert spaces $H_1$ and $H_2$ is said to be completely continuous if $y_n \rightharpoonup \bar{y}$ in $H_1$ implies $T y_n \rightarrow T \bar{y}$ in $H_2$. Weak continuity means $y_n \rightharpoonup \bar{y}$ in $H_1$ implies $T y_n \rightharpoonup T \bar{y}$ in $H_2$. Finally, a function $f\colon H \rightarrow \R$ is said to be weakly lower semicontinuous if  
\begin{equation*}
\mbox{if } v_k \rightharpoonup \bar{v} \mbox{ in } H, \; \mbox{ then } \liminf_{k \rightarrow \infty } f(v_k)  \geq f(\bar{v}). 
\end{equation*}

\subsection{Main results}\label{subsect:main_results} 
In this section, we study the asymptotic convergence of Algorithm~\ref{alg:PSG_Hilbert_Nonconvex_Unconstrained} for problems of the form \eqref{eq:main_prb}. We assume that the chosen step-sizes in Algorithm~\ref{alg:PSG_Hilbert_Nonconvex_Unconstrained} satisfy the (Robbins--Monro) conditions 
\begin{equation}\label{eq:Robbins-Monro-stepsizes}
 t_n \geq 0, \quad \sum_{n=1}^\infty t_n = \infty, \quad \sum_{n=1}^\infty t_n^2 < \infty.
\end{equation} 
To simplify the notation in this section, the inner product $(\cdot,\cdot)_H$ and the norm $\lVert \cdot \rVert_H$ on $H$ are denoted by $(\cdot,\cdot)$ and $\| \cdot \|$, respectively. We consider the following assumptions.
\begin{assumption}
\label{asu1}
Let $\{\mathcal{F}_n \}$ be a filtration, let $\{ u_n\}$ and $\{\xi_n \}$ be sequences of iterates and random vectors generated by Algorithm \ref{alg:PSG_Hilbert_Nonconvex_Unconstrained} with a stochastic gradient $\mathcal{G}\colon H \times \Xi \rightarrow H$. We assume \\
\subasu \label{asu1i} The sequence $\{ u_n\}$ is a.s.~contained in the bounded set $U \subset H$ and $u_n$ is adapted to $\mathcal{F}_n$ for all $n$.  \\
\subasu \label{asu1ii} On an open and convex set $V$ such that $U \subset V  \subset H$, the expectation $j$ is bounded below and belongs to $C_L^{1,1}(V)$.\\
\subasu \label{asu1iii} For all $n$, the $H$-valued random variable
\begin{equation}
\label{eq:def-rn}
r_n:= \E[\cG(u_n,\xi_n) | \mathcal{F}_n] - \nabla j(u_n)
\end{equation}
is adapted to $ \mathcal{F}_n$ and for $K_n:=\esssup_{\omega \in \Omega} \lVert r_n(\omega)\rVert$, the conditions \mbox{$\sup_{n} K_n<\infty$} and \mbox{$ \sum_{n=1}^\infty t_n K_n< \infty$} are satisfied.
\end{assumption}

In the following analysis, it will also be helpful to define the  $H$-valued random variable
\begin{equation}
\label{eq:def-wn}
\mathfrak{w}_n := \cG(u_n, \xi_n) - \E[\cG(u_n, \xi_n) | \mathcal{F}_n].
\end{equation}

In this section, we will make use of the following two results, which can be found in \cite[Theorem 9.4]{M2011} and \cite{RS}, respectively. We use the notation $\beta^{-}:=\max\{0,-\beta \}.$

\begin{lemma}[Almost sure convergence of quasimartingale]
\label{lemma:Quasimartingale_convergence_theorem}
Let $\{\mathcal{F}_n\}$ be a filtration and $v_n$ be a (real-valued) random variable adapted to $\mathcal{F}_n$. If 
$$\sup_n \E[v_n^{-}]<\infty \mbox{ and }
\sum_{n=1}^\infty \E\big[|\E[v_{n+1} - v_n| \mathcal{F}_n] |\big]  < \infty$$
 then $\{v_n \}$ converges a.s.~to a $P$-integrable random variable $v_\infty$ and it holds that $$\E [ | v_\infty |] \leq \liminf_n \E[|v_n|]<\infty.$$
\end{lemma}
\begin{lemma}[Convergence theorem for nonnegative almost supermartingales]\label{lemma:convegence_nonnegative_supermartingales}
Assume that $\{\mathcal{F}_n\}$ is a filtration and $v_n$,
$a_n$, $b_n$, $c_n$ nonnegative random variables adapted to $\{\mathcal{F}_n\}$. If
$$\mathbb{E} [v_{n+1}| \mathcal{F}_n ] \leq v_n(1 + a_n) + b_n - c_n \mbox{ a.s. }$$
and $\sum_{n=1}^{\infty} a_n < \infty, \quad \sum_{n=1}^{\infty} b_n < \infty$ a.s.,
then with probability one, $\lbrace v_n \rbrace $ is
convergent and $\sum_{n=1}^{\infty} c_n < \infty.$
\end{lemma}

Now, we will provide sufficient conditions for the boundedness condition in Assumption~\ref{asu1i}. We generalize the arguments from \cite{Bottou1998} to Hilbert-valued iterates.

\begin{lemma}
 \label{lem:sufficient-conditions-boundedness}
Suppose that $\{u_n \}$ is generated by Algorithm~\ref{alg:PSG_Hilbert_Nonconvex_Unconstrained} (with step-sizes given by \eqref{eq:Robbins-Monro-stepsizes}) and $u_n$ is $\mathcal{F}_n$-measurable for all $n$, where $\{\mathcal{F}_n \}$ is a filtration. Additionally, suppose Assumption~\ref{asu1iii} holds and there exists a $\gamma \geq 1$ such that 
\begin{equation}
\label{eq:angle-condition}
\inf_{\lVert u\rVert^2 \geq \gamma} (\nabla j(u),u)\geq 0
\end{equation}
and $A_k, B_k \geq 0$ (for $k=2,3,4$) such that
\[
\E[\lVert \mathcal{G}(u,\xi) \rVert^k ] \leq A_k + B_k \lVert u \rVert^k, \quad k=2,3,4.
\]
 Then $\{ u_n\}$ is a.s.~contained in a bounded set.
\end{lemma}

\begin{proof}
We define the sequence $f_n:=\varphi(\lVert u_n \rVert^2)$ using the function 
\[ \varphi(x):= \begin{cases}
0 & \text{if } x < \gamma\\
(x-\gamma)^2 & \text{if } x\geq \gamma.
\end{cases}
\]
One can verify that $\varphi$ satisfies the inequality
\[
\varphi(y) - \varphi(x) \leq \varphi'(x) (y-x) + (y-x)^2.
\]
Hence, we have with $g_n:=\mathcal{G}(u_n,\xi_n)$ and $u_{n+1}=u_n - t_n g_n$,
\begin{align*}
    f_{n+1} - f_n &\leq (-2t_n (u_n,g_n) + t_n^2 \lVert g_n\rVert^2) \varphi'(\lVert u_n\rVert^2) \\
    &\quad + 4t_n^2(u_n,g_n)^2 -4t_n^3 \lVert g_n\rVert^2(u_n,g_n) + t_n^4 \lVert g_n\rVert^4 \\
    & \leq (-2t_n (u_n,g_n) + t_n^2 \lVert g_n\rVert^2) \varphi'(\lVert u_n\rVert^2) + 4t_n^2 \lVert u_n\rVert^2 \lVert g_n\rVert^2 \\
    &\quad +4t_n^3 \lVert u_n\rVert \lVert g_n\rVert^3 + t_n^4 \lVert g_n\rVert^4
    \\
    &\leq (-2t_n (u_n,g_n) + t_n^2 \lVert g_n\rVert^2) \varphi'(\lVert u_n\rVert^2) 
      + 8 t_n^2 \lVert u_n\rVert^2 \lVert g_n\rVert^2 + 2 t_n^4 \lVert g_n\rVert^4.
\end{align*}
Using $\E[\lVert g_n \rVert^2|\mathcal{F}_n] = \E[\lVert \mathcal{G}(u_n,\xi)\rVert^2]$ and the fact that $u_n$ is $\mathcal{F}_n$-measurable, 
\begin{align*}
    \E[f_{n+1}|\mathcal{F}_n] - f_n &\leq  ( -2t_n (u_n,\nabla j(u_n))  + 2t_n\lVert u_n\rVert K_n 
    +t_n^2 \E[\lVert \mathcal{G}(u_n,\xi)\rVert^2]) \varphi'(\lVert u_n\rVert^2) \\
    &\quad + 8 t_n^2 \lVert u_n\rVert^2 \E[\lVert \mathcal{G}(u_n,\xi)\rVert^2] + 2 t_n^4 \E[\lVert \mathcal{G}(u_n,\xi)\rVert^4].
    \end{align*}
For the case $\lVert u_n\rVert^2 < \gamma$, notice that $\varphi'(\lVert u_n\rVert^2) = 0$. Using the growth conditions and \eqref{eq:angle-condition}, 
we get the existence of constants $A_0$ and $B_0$ such that (for $n$ large enough)
\begin{equation}
\label{eq:inequality-boundedness}
\begin{aligned}
\E[f_{n+1}|\mathcal{F}_n] - f_n 
 \leq 
 2t_n\lVert u_n\rVert K_n  \varphi'(\lVert u_n\rVert^2)+t_n^2 (A_0 + B_0 \lVert u_n\rVert^4).
 \end{aligned}
    \end{equation}
For the case $\lVert u_n\rVert^2 < \gamma$, 
\eqref{eq:inequality-boundedness} yields
\begin{equation}
\label{eq:inequality-boundedness2}
 \E[f_{n+1}|\mathcal{F}_n] -f_n \leq t_n^2 (A_0 + B_0 \gamma^2).
    \end{equation}
Otherwise, let $\lVert u_n \rVert \geq \gamma \geq 1$. Note that $B_0 \lVert u_n\rVert^4 \leq B\varphi(\lVert u_n\rVert^2)$ and additionally, $2\| u_n\| \varphi'(\| u_n\|^2) \leq C \varphi (\| u_n\|^2)$ for $B, C$ sufficiently large. Therefore, by \eqref{eq:inequality-boundedness}, we have 
\[
 \E[f_{n+1}|\mathcal{F}_n] -f_n \leq  Ct_n K_n f_n+t_n^2 (A + B f_n).
    \]
We have chosen $A$ large enough to handle the case \eqref{eq:inequality-boundedness2}. 
Rearranging, 
\[
 \E[f_{n+1}|\mathcal{F}_n] \leq f_n(1+ Ct_n K_n + t_n^2 B) +t_n^2 A.
    \]
Since $\sum_{n=1}^\infty t_n K_n < \infty$ and $\sum_{n=1}^\infty t_n^2 < \infty$, we have that $\{f_n \} $ is almost surely convergent by Lemma~\ref{lemma:convegence_nonnegative_supermartingales}, implying that $\{\lVert u_n\rVert\}$ and hence $\{u_n\}$ are almost surely bounded.
\end{proof}

Next, we provide sufficient conditions for verifying measurability. Typically, this property is easily argued in the finite-dimensional setting, but in the infinite-dimensional case, one must distinguish between weak and strong measurability. A workaround is to assume separability of the underlying space, so that these notions coincide. 
\begin{lemma}
\label{lem:measurability}
Suppose $K$ is a closed subset of a (real) separable Hilbert space $H$, $\Xi$ is a (real) complete separable metric space, and let $\{\mathcal{F}_n\}$ defined by $\mathcal{F}_n = \sigma(\xi_1, \dots, \xi_{n-1})$ be the natural filtration generated by the stochastic process $\{ \xi_n\}.$ Suppose $\cG\colon H \times \Xi \rightarrow H$ and $\nabla j\colon V \rightarrow H$ are continuous on $K \times \Xi$ and $K$, respectively, and $\{u_n\}$, defined by the recursion in Algorithm~\ref{alg:PSG_Hilbert_Nonconvex_Unconstrained}, is contained in $K$ for all $n.$ Then $u_n$, as well as the functions $r_n$ and $\mathfrak{w}_n$, defined by \eqref{eq:def-rn} and \eqref{eq:def-wn}, respectively, are adapted to $\mathcal{F}_n$ for all $n.$
\end{lemma}

\begin{proof}
Note that $K$, as a closed subset of a (separable) Hilbert space, is a complete (separable) metric space when equipped with the norm restricted to the subset. Now, we recall that continuity in combination with separability of $H$ and $\Xi$ implies the superpositional measurability of $\cG$, i.e., if $\omega \mapsto u(\omega)$ and $\omega \mapsto \xi(\omega)$ are measurable, then so is the mapping $\omega \mapsto \cG(u(\omega),\xi(\omega))$; cf.~\cite[Lemma 8.2.3]{Aubin1990}. Additionally, for a Banach space $X$ and function $f: \Omega \rightarrow X$, we recall the following fact: if $f$ is $\mathcal{F}_n$-measurable,  
then $f$ is $\mathcal{F}_m$-measurable for any $m\geq n$ on account of the inclusion $\mathcal{F}_n \subset \mathcal{F}_{n+1}$ for every $n.$

Now, we prove the statement of the theorem by induction. Notice that $u_1\in K$ is deterministic by definition in the algorithm. By construction, $\xi_{n}$ is  $\mathcal{F}_{n}$-measurable. Suppose $u_{n}$ is $\mathcal{F}_{n}$-measurable. Hence the mappings $\omega \mapsto \cG(u_n(\omega), \xi_n(\omega))$ and $\omega \mapsto \nabla j(u_n(\omega))$ are $\mathcal{F}_{n}$-measurable. By definition, the conditional expectation $\E[\cG(u_n,\xi_n)|\mathcal{F}_n]$ is $\mathcal{F}_n$-measurable. The $\mathcal{F}_n$-measurability of $r_n$ and $\mathfrak{w}_n$ follows as does the $\mathcal{F}_{n+1}$-measurability of $u_{n+1} = u_n -t_n \cG(u_n,\xi_n).$
\end{proof}

\begin{remark}
\label{rem:measurability}
In some applications, it may be easier to verify measurability using the function $G(u,\omega)$ as opposed to $\cG(u,\xi)$, where $G(u,\omega)$ is defined by $G(u,\omega)=\mathcal{G}(u,\xi(\omega))$. In this case, it is sufficient to assume (in addition to $u\mapsto \nabla j(u)$ being  continuous) that $G$ is Carath\'eodory, i.e., $K \ni u \mapsto G(u,\omega)$ is continuous for almost all $\omega$ and $\omega \mapsto G(u,\omega)$ is measurable for all $u$. It is clear that the proof above is simplified since the mapping $\omega \mapsto G(u(\omega),\omega)$ is measurable. Under these conditions, it therefore follows that $u_n, r_n$, and $\mathfrak{w}_n$ are adapted to $\mathcal{F}_n$ for all $n.$
\end{remark}

\begin{assumption}\label{assumption:well-posedness-unconstrained}
There exists a function $M\colon H \rightarrow [0,\infty)$, that is bounded on bounded sets, such that 
 $$\E[\lVert \cG(u,\xi) \rVert^2] \leq M(u) \quad \forall u\in H. $$
\end{assumption}

\begin{remark}
\label{remark-assumptions-unconstrained}
Since we assume that $\{ u_n\}$ is a.s.~bounded by Assumption~\ref{asu1i}, Assumption~\ref{assumption:well-posedness-unconstrained} covers standard conditions on the growth of the second moment, such as the conditions $\E[\lVert \cG(u,\xi) - \nabla j(u) \rVert^2] \leq \sigma$ or $\E[\lVert \cG(u,\xi) \rVert^2] \leq M_1 + M_2 \lVert u \rVert^2$ for given constants $\sigma, M_1, M_2.$
\end{remark}

The convergence of the algorithm is proven next. While convergence is covered by the theory in \cite{GS2021}, a Sard-type assumption is needed there that is difficult to verify in the infinite-dimensional setting. Given the smoothness assumed here, a more direct proof is possible by generalizing arguments from \cite[Section 4.3]{Bottou} made for the finite-dimensional setting. This generalization includes the handling of different topologies on the Hilbert space as well as the inclusion of the additive bias term $r_n$, which is important for applications with PDEs, where solutions can only be approximated; see \cite{GW}. 

\begin{theorem}
\label{theorem:convergence-algorithm1}
Let Assumptions~\ref{asu1} and \ref{assumption:well-posedness-unconstrained} hold and let $\{u_n\}$ be a sequence generated by Algorithm \ref{alg:PSG_Hilbert_Nonconvex_Unconstrained} with the step-size rule~\eqref{eq:Robbins-Monro-stepsizes}. Then, the sequence $\{ j(u_n)\}$ converges a.s.~and $\liminf_{n \rightarrow \infty}  \lVert \nabla j(u_n) \rVert = 0$ a.s. Furthermore,
\begin{enumerate}
\item If $F(u):= \lVert \nabla j (u) \rVert^2$ satisfies $F \in C^{1,1}_{L_F}(V)$, then $\lim_{n \rightarrow \infty} \nabla j(u_n) = 0$ a.s. In particular, every strong accumulation point of $\{u_n\}$ is stationary point for $j$ with probability one.
\item If, in addition to the assumptions above, the map $u \mapsto \| \nabla j(u) \|$ is weakly lower semicontinuous, 
then weak accumulation points of $\{u_n\}$ must be a stationary points for $j$ with probability one.
\end{enumerate}
\end{theorem}

\begin{proof} 
By Assumption~\ref{asu1ii}, there exists a finite $\bar{j} := \inf_{u \in V} j(u)$. Without loss of generality, we assume $j \geq 0$ on $V$; otherwise one could make the same arguments for $\tilde j:=j-\bar{j} \geq 0.$ 
Since $j \in C_L^{1,1}(V)$, by the estimate in \eqref{eq:taylor_est} and with $g_n:=\cG(u_n,\xi_n)$, it follows that
\begin{equation}
\label{eq:first-inequality-unconstrained-proof}
 j(u_{n+1})  \leq j(u_n) -t_n ( \nabla j(u_n), g_n ) +\frac{Lt_n^2}{2}\lVert g_n \rVert^2.
 \end{equation}
By Assumptions~\ref{asu1iii}, \ref{assumption:well-posedness-unconstrained} and \eqref{eq:def-wn}, we have {$\E[g_n|\mathcal{F}_n] = \nabla j(u_n) + r_n $} and $\E[\lVert g_n\rVert^2 | \mathcal{F}_n] = \E[\lVert \cG(u_n,\xi)\rVert^2] \leq M(u_n).$ Notice that for all $n$,
$$\lVert \nabla j(u_n)\rVert \leq \lVert \nabla j(u_n) - \nabla j(u_1) \rVert + \lVert \nabla j(u_1) \rVert \leq L \, \text{diam}(U) + \lVert \nabla j(u_1) \rVert=:M_1,$$
where $\text{diam}(U)$ denotes the diameter of the subset $U\subset H$, which is bounded by Assumption~\ref{asu1i}.
Thus, taking the conditional expectation with respect to $\mathcal{F}_n$ on both sides of \eqref{eq:first-inequality-unconstrained-proof} and using that $u_n$ and $r_n$ are $\mathcal{F}_n$-measurable, we get
 \begin{equation}
 \label{martingale-inequality-unconstrained}
 \begin{aligned}
  \E[j(u_{n+1})|\mathcal{F}_n] &\leq j(u_n) - t_n ( \nabla j(u_n), \E[g_n|\mathcal{F}_n] ) +\frac{Lt_n^2}{2}\E[\lVert g_n \rVert^2 | \mathcal{F}_n]\\
  & \leq j(u_n) - t_n ( \nabla j(u_n),\nabla j(u_n) +r_n ) +\frac{L  M(u_n) t_n^2}{2}\\
  & \leq j(u_n) - t_n \lVert \nabla j(u_n) \rVert^2 + M_1 t_n K_n + \frac{L M(u_n) t_n^2}{2}.
  \end{aligned}
 \end{equation}
Observe that by Assumption~\ref{asu1i}, the sequence $\{ u_n\}$ belongs to the bounded set $U$ almost surely, and by Assumption~\ref{assumption:well-posedness-unconstrained}, $M(\cdot)$ is bounded on bounded sets. Therefore, $M(u_n)$ is uniformly bounded for all $n$ with probability one. Now, assigning $v_n = j(u_n)$,  $a_n = 0$,  $b_n = M_1 t_n K_n + (L M(u_n) t_n^2)/2$, and $c_n = t_n  \lVert \nabla j(u_n) \rVert^2,$
we get by Lemma \ref{lemma:convegence_nonnegative_supermartingales} that $\{j(u_n)\}$ converges and furthermore,
$\sum_{n=1}^\infty t_n \lVert \nabla j(u_n) \rVert^2 < \infty$
with probability one. Due to the step-size condition \eqref{eq:Robbins-Monro-stepsizes}, it follows that $\liminf_{n \rightarrow \infty}  \lVert \nabla j(u_n) \rVert = 0$ a.s. This proves the first claim.

For the claim in $1.$, we assume that $F\in C^{1,1}_{L_F}(V)$ and we first show that 
\begin{equation}
\label{eq:unconstrained-proof-summability}
\sum_{n=1}^\infty t_n \E[\lVert \nabla j(u_n) \rVert^2] < \infty.
\end{equation}
This can be seen by taking the expectation and summing on both sides of \eqref{martingale-inequality-unconstrained}. Indeed, after rearranging, we get
\begin{equation}
\label{eq:partial-sum-unconstrained-proof}
 \begin{aligned}
\sum_{n=1}^N t_n \E[\lVert \nabla j(u_n) \rVert^2] &\leq \sum_{n=1}^N \left( \E[j(u_n)] - \E[j(u_{n+1})] + M_1 t_n K_n + \frac{LM(u_n) t_n^2}{2}\right)\\
 &\leq \E[j(u_1)] - \bar{j} + \sum_{n=1}^N \left( M_1 t_n K_n + \frac{LM(u_n)t_n^2}{2} \right).
 \end{aligned}
\end{equation}
By Assumption~\ref{asu1iii} and the condition \eqref{eq:Robbins-Monro-stepsizes}, the right-hand side of \eqref{eq:partial-sum-unconstrained-proof} is bounded as $N\rightarrow \infty$. Therefore, since the left-hand side is monotonically increasing in $N$ and bounded above, \eqref{eq:unconstrained-proof-summability} must hold.

Now, with $F(u) = \lVert \nabla j(u) \rVert^2$, we get by the estimate in \eqref{eq:taylor_est} that
\begin{equation}
 \label{eq:LS-h}
\begin{aligned}
 F(u_{n+1})-F(u_n) &\leq -t_n ( \nabla F(u_n), g_n ) +\frac{L_Ft_n^2}{2}\lVert g_n \rVert^2, \\
 F(u_{n})-F(u_{n+1}) &\leq t_n ( \nabla F(u_n), g_n ) +\frac{L_Ft_n^2}{2}\lVert g_n \rVert^2.
 \end{aligned}
 \end{equation}
By the chain rule, taking
$E(u): = \lVert u \rVert^2$, $H(u):= \nabla j(u)$, we have that $F(u) = E(H(u))$ and, in turn, $F'(u) = E'(H(u))H'(u)$. 
Since $E'(u)v = ( 2 u,v)$ and $H'(u)v=\nabla^2j(u)v$, we get that 
\begin{equation}\label{eq:derivative_F}
F'(u)v  = ( 2 \nabla j(u), \nabla^2 j(u) v) = 2( (\nabla^2 j(u))^* \nabla j(u), v ).
\end{equation}
Additionally, by Lipschitz continuity of $\nabla j$ and boundedness of $\{u_n\}$, we have that $\lVert (\nabla^2 j(u_n))^* \rVert_{\mathcal{L}(H,H)} = \lVert \nabla^2 j(u_n) \rVert_{\mathcal{L}(H,H)} \leq L$. Taking the conditional expectation with respect to $\mathcal{F}_n$ on both sides of \eqref{eq:LS-h}, we obtain
\begin{equation}
\label{eq:LS-h-martingale}
\begin{aligned}
 &|\E[F(u_{n+1})|\mathcal{F}_n] - F(u_n)|\\
 & \quad\quad \leq \left|-t_n ( \nabla F(u_n), \E[g_n|\mathcal{F}_n] ) +\frac{L_Ft_n^2}{2} \E[\lVert g_n \rVert^2 | \mathcal{F}_n]\right|\\
 & \quad\quad \leq |-t_n ( 2 (\nabla^2 j(u_n))^* \nabla j(u_n), \nabla j(u_n) +r_n) | +\frac{L_F M(u_n) t_n^2}{2} \\
 &  \quad\quad \leq 2 L t_n \lVert \nabla j(u_n) \rVert^2  + 2 L M_1 t_n K_n + \frac{L_F M(u_n) t_n^2}{2}.
 \end{aligned}
\end{equation}

Now we can verify the conditions of Lemma~\ref{lemma:Quasimartingale_convergence_theorem} with $v_n := F(u_n)$. By \eqref{eq:LS-h-martingale}, we have
\begin{align*}
&\sum_{n=1}^\infty \E[|\E[v_{n+1}|\mathcal{F}_n] - v_n|]\\
& \quad \quad \leq  \sum_{n=1}^\infty \E \left[ 2 L t_n \lVert \nabla j(u_n) \rVert^2  + 2 L M_1 t_n K_n + \frac{L_F M(u_n) t_n^2}{2} \right],
\end{align*}
where the right-hand side is finite by Assumption~\ref{asu1iii}, the condition \eqref{eq:Robbins-Monro-stepsizes}, and \eqref{eq:unconstrained-proof-summability}. Naturally, $\sup_{n} \E[v_n^{-}] = 0$. Thus we get by Lemma~\ref{lemma:Quasimartingale_convergence_theorem} that $F(u_n)=\lVert \nabla j(u_n)\rVert^2$ converges a.s., which by the first part of the proof can only converge to zero. We obtain that {$\lim_{n\rightarrow \infty}\nabla j(u_n) =0$} a.s. 

To show claim 2 of this theorem, let us observe that since $\lbrace u_n \rbrace$ is bounded and $H$ is a Hilbert space, there exists a subsequence $\lbrace u_{n_k} \rbrace$ and a limit point $\bar{u}$ such that $u_{n_k} \rightharpoonup \bar{u}$ in $H$. Moreover, by the first claim of this theorem, we know that $\lim_{n\rightarrow \infty} \| \nabla j(u_{n}) \|=0$.
Hence, fix any subsequence $\lbrace u_{n_k} \rbrace$  of the process produced by the Algorithm \ref{alg:PSG_Hilbert_Nonconvex_Unconstrained} which admits a limit point $\bar{u}$ such that $u_{n_k} \rightharpoonup \bar{u}$ in $H$. Weak lower semicontinuity of the mapping $u \mapsto \| \nabla j(u) \|$ ensures that \begin{equation*}\label{eq:for_claim3}
 \| \nabla j(\bar{u})\| \leq    \liminf_{k\rightarrow \infty} \| \nabla j(u_{n_k})\| =0 
\end{equation*}    
and thus the fact that $\bar u$ is a stationary point with probability one.
 \end{proof}

 \begin{remark}
The above theorem yields sufficient conditions for weak or strong cluster points of $\lbrace u_n \rbrace$ to be stationary points. Exploiting additional properties of Problem \eqref{eq:main_prb} can guarantee stronger results, such as (local) convergence to a local optimum. Let us give an example. Under the assumptions for claim 2 of Theorem \ref{theorem:convergence-algorithm1}, suppose additionally that $\bar u$ is a local optimum with $\|\nabla j(\bar u)\|=0$ and there exists a neighborhood $B_r(\bar u)$ of $\bar u$ such that $\|\nabla j(u)\|\neq 0$ for all $u\in B_r(\bar u),\; u \neq \bar u$. Of the set of all possible random sequences  generated by Algorithm \ref{alg:PSG_Hilbert_Nonconvex_Unconstrained}, consider a fixed sequence $\{ u_n \}$ having the following property: $u^*$ is a weak cluster point of $\lbrace u_n \rbrace$ and asymptotically, $\lbrace u_n \rbrace$ stays in $B_r(\bar u)$. If $u\mapsto \| \nabla j(u)\|$ is weakly lower semicontinuous, then by Theorem \ref{theorem:convergence-algorithm1}, $u^*$ must coincide with the local optimum $\bar u$. 
\end{remark}

The condition in the first claim of Theorem~\ref{theorem:convergence-algorithm1} can for instance be verified using the following lemma.
\begin{lemma}
\label{lemma:twice-differentiability}
Assume that $j$ is twice Fr\'echet differentiable with a Lipschitz second-order derivative that is bounded on $V$. Then, the condition $u\mapsto F(u) = \lVert \nabla j (u) \rVert^2$ belonging to $C^{1,1}_{L_F}(V)$ is fulfilled.  
\end{lemma}
\begin{proof}
One can easily deduce the claim by the fact that $F(u)=E(H(u))$, where $H(u):=\nabla j(u)$ and $E(u)=\| u \|^2$ and remembering \eqref{eq:derivative_F}. 
Given $u_1, u_2 \in V \subset H$,
\begin{align*}
&\| 2 \nabla^2 j(u_1)^* \nabla j(u_1)-2 \nabla^2 j(u_2)^* \nabla j(u_2) \| \\ 
&\leq \| 2 \nabla^2 j(u_1)^* \nabla j(u_1)-2 \nabla^2 j(u_1)^* \nabla j(u_2) \|  + \| 2 (\nabla^2 j(u_1)^* - \nabla^2 j(u_2)^*) \nabla j(u_2) \| \\
& \leq 2 \|  \nabla^2 j(u_1) \| \| \nabla j(u_1)-  \nabla j(u_2) \|  + 2 \| \nabla^2 j(u_1)^* - \nabla^2 j(u_2)^*\|  \| \nabla j(u_2) \|.
\end{align*}
Now, let $c_V$ be the Lipschitz constant of $\nabla^2 j$ on $V$ and $L$ be a constant such that $\sup_{u\in V} \| \nabla^2 j (u) \|\leq L$. One can prove in a straightforward way that $\|\nabla j(u_1) -\nabla j(u_2)\|\leq L \| u_1 - u_2\|$ for any $u_1,u_1 \in V$ and also  $\| \nabla j(u) \|$ must be bounded on $V$. Hence,
\begin{align*}
&\| 2 \nabla^2 j(u_1)^* \nabla j(u_1)-2 \nabla^2 j(u_2)^* \nabla j(u_2) \|\\
& \quad  \leq 2 L^2 \| u_1- u_2 \| + 2 c_V \| u_1-u_2 \| \sup_{u\in V} \| \nabla j(u) \|.
\end{align*}
This yields the claim of this lemma.
\end{proof}

From Lemma~\ref{lemma:twice-differentiability}, one sees that the question of whether the mapping $u\mapsto F(u)$ belongs to $C^{1,1}_{L_F}$ requires the analysis of the second-order derivative of $j$; this argument is certainly quite involved in the application to PDE-constrained optimization with uncertainties. Moreover, this statement only provides stationarity of strong accumulation points of sequences generated by the stochastic gradient method. This is rather inconvenient, since in a bounded sequence on a Hilbert space one can usually only hope to obtain weakly convergent subsequences. As we will demonstrate in Section~\ref{sec:OCforPDEs}, in certain applications it is in fact possible to prove weak lower semicontinuity of the mapping $u \mapsto \| \nabla j(u) \|$ even if this map is nonconvex.
Briefly, this occurs when the gradient can be split into a completely continuous part and a weak sequential continuous part. In this case, one has the stronger result, namely stationarity of weak accumulation points.

For deterministic problems, a backtracking procedure is frequently employed in combination with the gradient descent method to guarantee descent. The following example demonstrates how the Armijo backtracking rule when combined with stochastic gradients fails in minimizing a function over the real line.

\begin{example}[Armijo line search diverges for SGM]
\label{example-line-search}
Let $\cJ\colon \R \times \Xi \rightarrow \R, (u,\xi) \mapsto (u+\xi)^2$. A stochastic gradient for this function is $\nabla_u \cJ(u,\xi) = 2(u+\xi).$ Now consider a random variable taking two values with equal probability, say $\xi = \pm 1$. The minimum of $\E[\cJ(u,\xi)] = \tfrac{1}{2}((u+1)^2+(u-1)^2)$ is attained at $u=0$. 
 The Armijo rule with one randomly chosen sample $\xi_n$ amounts to finding the smallest integer $m_n$ such that for $\beta > 0$, $\alpha_n=\beta t^{m_n}$, the descent direction $p_n = -\nabla_u \cJ(u_n,\xi_n)$ satisfies the following Wolfe condition for $0 < c < 1$ and $0 < t < 1$:
 $$\cJ(u_n+\alpha_n p_n,\xi_n) \leq \cJ(u_n,\xi_n) + c \alpha_n \nabla_u \cJ(u_n,\xi_n) \cdot p_n.$$
For the choice of function, this condition is equivalent to
$$(u_n - 2 \alpha_n (u_n + \xi_n) + \xi_n)^2 \leq (u_n+\xi_n)^2 - 4 c \alpha_n (u_n + \xi_n)^2,$$
so after simplifying, we see that the Wolfe condition is satisfied for the smallest $m_n$ such that
$\alpha_n \leq 1 - c.$
Let $\alpha \equiv \alpha_n$ be this step-size, which is independent of $n$. Suppose now $\{ u_n \}$ were to converge to the minimizer. Assuming $\varepsilon < \tfrac{\alpha}{1-\alpha}$, a simple argument shows that $|u_n| < \varepsilon$ implies $|u_{n+1}| > \varepsilon$ for all $n$, which contradicts this convergence. Indeed, for $\xi=-1$,
\begin{align*}
 u_{n+1} = u_n - \alpha \nabla_{u} \cJ(u_n, -1) \geq -\varepsilon (1-2\alpha) + 2 \alpha > \varepsilon
\end{align*}
by the assumption on $u_n$ and $\varepsilon$. Additionally, we have for $\xi=1$ that
\begin{align*}
 u_{n+1} = u_n - \alpha \nabla_{u} \cJ(u_n, 1) \leq  \varepsilon (1-2\alpha) - 2 \alpha < -\varepsilon.
 \end{align*}
 Therefore, the next iterate will always leave the $\varepsilon$-neighborhood of the true optimum.
\end{example}
This examples underscores the need for some way to reduce variance along the optimization procedure; as we already demonstrated in Theorem~\ref{theorem:convergence-algorithm1}, that is accomplished by means of the Robbins--Monro rule (2.3). However, there are other modifications of the basic method Algorithm~\ref{alg:PSG_Hilbert_Nonconvex_Unconstrained} that achieve the same effect, for instance by (indefinitely) increasing batch sizes in the stochastic gradient \cite{Wardi, Wardi2}. Adaptive step-size choices were also studied in \cite{George, Yousefian2012}.
Recent advances to adaptively sample in combination with a line-search procedure (in finite dimensions) can be found in \cite{Bollapragada, PS2020}.

\subsection{Asymptotic almost sure convergence rates}
In this subsection we derive some almost sure convergence estimates for Algorithm \ref{alg:PSG_Hilbert_Nonconvex_Unconstrained}. Set $j^* = \inf_{u \in V} j(u)$. The result below is closed to the one developed in previous papers on convergence analysis for stochastic gradient descent methods in the non-convex finite-dimensional setting (see, e.g., \cite{KR,SGD}). However, the analysis there is based on a so-called \emph{expected smoothness assumption} \cite[Assumption 2]{KR}, which seems to be stronger than our setting. 
Moreover, our result differs from \cite{KR,SGD} due to the presence of bias $r_n$ in the stochastic gradient. 
\begin{theorem}
\label{thm:convergence-rate}
Let Assumptions~\ref{asu1} and \ref{assumption:well-posedness-unconstrained} hold and let $\{u_n\}$ be a sequence generated by Algorithm \ref{alg:PSG_Hilbert_Nonconvex_Unconstrained} with the step-size rule~\eqref{eq:Robbins-Monro-stepsizes}. 
If we additionally assume that $t_n$ is decreasing and
\begin{equation}\label{eq:step_size_rules2}
\sum_{j=1}^\infty \frac{t_j}{\sum_{k=1}^j t_k} =\infty,
\end{equation}
then, almost surely,
\begin{equation}\label{eq:conv_est_g}
    \min_{t=1,\dots, n} \| \nabla j(u_t)\|^2= o \left( \frac{1}{\sum_{j=1}^{n} t_j} \right).
\end{equation}
\end{theorem}
\begin{proof}
We first define for all $n\in \mathbb{N}$, $n>0$,
\begin{equation*}
\eta_n= \frac{2 t_n}{\sum_{j=1}^n t_j}, \quad T_1 =\| \nabla j(u_1) \|^2, \quad T_{n+1}= (1-\eta_n)T_n+\eta_n \| \nabla j(u_n)\|^2.
\end{equation*}
Note that $\eta_1=2$ and $\eta_n \in[0,1]$ for $n>1$ since $t_n$ is decreasing. Indeed, 
$$0\leq \eta_n = \frac{2 t_n}{\sum_{j=1}^n t_j}\leq \frac{2 t_n}{n t_n}\leq 1 \textit{ for } n>1.$$
Moreover,
\begin{equation*}
  2  t_n \| \nabla j(u_n) \|^2 = \sum_{j=1}^n t_j T_{n+1} + t_n T_n - \sum_{j=1}^{n-1} t_j T_n. 
\end{equation*}
Plugging into \eqref{martingale-inequality-unconstrained}, we obtain
 \begin{equation*}
 \begin{aligned}
  & \E\Big[j(u_{n+1}) + \frac{1}{2} \sum_{j=1}^n t_j T_{n+1}\Big|\mathcal{F}_n\Big]= \E[j(u_{n+1})|\mathcal{F}_n ] + \frac{1}{2} \sum_{j=1}^n t_j T_{n+1} \\ &\qquad\leq j(u_n) - \frac{1}{2} t_n T_n + \frac{1}{2}\sum_{j=1}^{n-1} t_j T_{n} + M_1 t_n K_n + \frac{L M(u_n) t_n^2}{2}.
  \end{aligned}
 \end{equation*}
Hence, by Lemma \ref{lemma:convegence_nonnegative_supermartingales} $$\Big\lbrace j(u_n)+ \frac{1}{2} \sum_{j=1}^{n-1} t_j T_{n}\Big\rbrace  \text{ converges with probability one and } \sum_{n=1}^\infty t_n T_n < \infty.$$ Since $\lbrace j(u_n)\rbrace$ converges almost surely by Theorem~\ref{theorem:convergence-algorithm1}, then  $\lbrace \sum_{j=1}^{n-1} t_j T_n\rbrace$ does so as well. Moreover, the above claim yields $$\lim_{n\rightarrow \infty} \frac{t_n}{\sum_{j=1}^{n-1} t_j} \sum_{j=1}^{n-1}t_j T_n=\lim_{n\rightarrow \infty} t_n T_n =0. $$
Using \eqref{eq:step_size_rules2}, $\lim_{n\rightarrow \infty} \sum_{j=1}^{n-1}t_j T_n=0$, that is, $T_n= o \left(  \frac{1}{\sum_{j=1}^{n-1} t_j} \right)$ almost surely. Proceeding by induction and using the definition of the sequence $T_n$ and the fact that $\eta_n \in [0,1]$ for $n>1$, one can easily verify that for any $n>1$ there exists a sequence  $\lbrace \tilde{\eta}_k \rbrace$ in $[0,1]$ such that $\sum_{j=1}^n \tilde{\eta}_j =1$ and $T_n=\sum_{k=1}^{n-1}\tilde{\eta}_k \| \nabla j(u_k)\|^2$. Moreover, 
\begin{align*}
&T_2=\| \nabla j(u_1)\|^2,\\
&T_n \geq \sum_{k=1}^{n-1}\tilde{\eta_k} \min_{t=1,\dots,n-1}  \| \nabla j(u_t)\|^2 = \min_{t=1,\dots,n-1} \|\nabla j(u_t)\|^2 \geq 0, \textit{ for } n>2.
\end{align*}
Thus, $T_n \geq \min_{t=1,\dots,n-1} \|\nabla j(u_t)\|^2 \geq 0$ holds for any $n>1$. This yields \eqref{eq:conv_est_g} and ends the proof.
\end{proof}
As noted in \cite{SGD}, the step-size condition \eqref{eq:step_size_rules2} covers many standard Robbins--Monro step-sizes, such as those of the form $t_n = \tfrac{\theta}{n^s}$ with $\theta>0$ and $s\in (\tfrac{1}{2},1].$

\section{PDE-constrained optimization problems under uncertainty}\label{sec:OCforPDEs}
In \cite{GS2021}, we studied proximal stochastic gradient methods applied to a problem with a specific semilinear elliptic PDE. In this section, we show that more general semilinear elliptic PDEs are permissible than shown in \cite{GS2021}. As a main contribution, we verify the assumptions for the convergence of the stochastic gradient method from Section~\ref{sub:stoch_grad_method} for this class of problems.

We consider the optimal control problem
\begin{equation*}
    \label{eq:objective}
    \min_{u \in U} \, \{ \E[\tilde{J}(y(\cdot), \cdot)] + \rho(u) \}
\end{equation*}
where $y=y(\omega)$ almost surely solves the operator equation
\begin{equation}
\label{eq:PDE_constraint}
e(u,y,\omega):=\mathcal{A}(\omega)y+\mathcal{N}(y,\omega)-\mathcal{B}(\omega) u - \mathfrak{b}(\omega) =0.\end{equation}
This setting is a risk-neutral version of the problem considered in \cite{KS}. In \cite{KS}, a reflexive Banach space $U$ is given for the control space and $Y=H^1(D)$ is the state space ($H^1(D)$ being the separable Hilbert space of $L^2(D)$ functions whose weak derivatives belong to $L^2(D)$). To obtain gradients and measurability of the iterates of Algorithm~\ref{alg:PSG_Hilbert_Nonconvex_Unconstrained} according to Lemma~\ref{lem:measurability}, we make the stronger assumption that $U=H$ is in fact a separable Hilbert space and $Y \hookrightarrow U = U^* \hookrightarrow Y^*$ is a Gelfand triple. The random operators appearing in \eqref{eq:PDE_constraint} are defined by $\cA\colon \Omega \rightarrow \mathcal{L}(Y,Y^*)$, $\mathcal{N}\colon Y\times \Omega \rightarrow Y^*$, $\cB\colon \Omega \rightarrow \mathcal{L}(U,Y^*)$, and $\mathfrak{b}\colon \Omega \rightarrow Y^*$.

The operator equation~\eqref{eq:PDE_constraint} covers the case \eqref{eq:PDEs_intro} given in the introduction with the choice $U=L^2(D)$ and the definitions
\begin{align*}
&\langle \cA(\omega)y,v\rangle_{Y^*,Y}:= \int_D \lbrace k(x,\omega)\nabla y(x) \cdot \nabla v(x) +c(x,\omega)y(x) v(x)\rbrace\D x, \nonumber \\
&\langle \cN(y,\omega),v \rangle_{Y^*,Y}:=\int_D N(y,\omega,x)v(x) \D x,\\ &\langle \cB(\omega)u,v\rangle_{Y^*,Y}:= \int_D [B(\omega)u](x) v(x) \D x , \quad \langle \cb(\omega),v \rangle_{Y^*,Y}:=\int_D b(x,\omega) v(x) \D x. \nonumber
\end{align*}

Throughout this section, we assume the analytical framework given in \cite{KS} applies. In particular, we use the same set of assumptions; for the reader's convenience, these are repeated in Assumptions~\ref{hp1}--Assumption~\ref{objective-function} in the Appendix. These assumptions ensure, among other things, that a unique solution to \eqref{eq:PDE_constraint} exists and belongs to $L^q(\Omega,Y)$ for $q=\tfrac{s\gamma}{1+s/t}$, where $s$ and $t$ are defined in Assumption~\ref{assumption_diff}. In particular, the control-to-state map $S(\omega)\colon U \rightarrow Y$ is well-defined and we can define the reduced objective function
\begin{equation*}
    J(u,\omega):=\tilde{J}(S(\omega)u,\omega)+\rho(u)
\end{equation*}
and the expectation $j(u):=\E[J(u,\cdot)].$


\subsection{Analysis of the optimal control problem} \label{sec:analysis_PDE_prob}
In this section, we verify the assumptions used in Theorem~\ref{theorem:convergence-algorithm1}, which ensures that the stochastic gradient method converges for our class of optimal control problems with semilinear PDE constraints. The next proposition concerns the smoothness of the objective function and the definition of the stochastic gradient. This analysis differs somewhat from \cite{Kouri2016, KS} since risk-averse objectives are generally nonsmooth and smoothness of the objective, which is needed to obtain a gradient, is not shown there. 
Here, $y_{\tilde{u}}$ and $p_{\tilde{u}}$ denote the solutions to \eqref{eq:PDE_constraint} and \eqref{eq:adjoint_p} with $u=\tilde{u}$, respectively.
\begin{proposition}\label{proposition_conditions_limit_points}
Let Assumptions \ref{hp1}--\ref{objective-function} hold. Then, a stochastic gradient $G\colon U\times \Omega \rightarrow U$ is given by  
\begin{equation}
G(u,\omega)= \nabla \rho( u) -  \mathcal{B}^*(\omega) p,
\label{eq:stoch-gradient}
\end{equation}
where $p=p_u(\omega) \in Y$ is almost surely the solution of the adjoint equation 
\begin{equation}\label{eq:adjoint_p}
(\cA^*(\omega) + \cN^*_y(y_u(\omega), \omega)) p= -D_y \tilde{J}(y_u(\omega),\omega)
\end{equation} 
and $y=y_u(\omega)$ is almost surely the solution to \eqref{eq:PDE_constraint}.
With the random variable $C\colon \Omega \rightarrow [0,\infty)$ from Assumption~\ref{hp1} we have for all $u$:
\begin{align}
&\| y_u(\omega) \|_{Y}^{\gamma}\leq C^{-1}(\omega) [ \| \cb(\omega)\|_{Y^*}+ \| \cB(\omega) \|_{\mathcal{L}(U,Y^*)} \| u \|_{U}] \quad \text{a.s.}, \label{eq:estimate_dual_arc_1} \\
&\| p_u(\omega) \|_{Y}^{\gamma}\leq C^{-1}(\omega) \| D_y \tilde{J}(y_u(\omega),\omega)\|_{Y^*} \quad \text{a.s.} \label{eq:estimate_dual_arc}
\end{align}
For arbitrary $u_1$ and $u_2$ belonging to $U$, we have
\begin{equation}\label{eq:Lipschitz_y}
    \| y_{u_1}(\omega) -y_{u_2}(\omega) \|_{Y}^{\gamma} \leq C^{-1}(\omega)  \| \cB(\omega) \rVert_{\mathcal{L}(U, Y^*)} \| u_1-u_2\|_{U} \quad \text{a.s.}
\end{equation} 
\end{proposition}
\begin{proof}
To obtain the stochastic gradient, one can argue using the adjoint approach on the reduced objective (cf.~\cite[Section 1.6.2]{Hinze}): the adjoint equation is given by 
\begin{equation*}
    (D_y e(u,y_u(\omega),\omega))^* p = -D_y \tilde{J}(y_u(\omega),\omega)
\end{equation*}
yielding \eqref{eq:adjoint_p}. The parametrized derivative of the reduced objective $J$ is given by
\begin{equation*}
    D_u J(u,\omega) = \rho'(u)+(D_u e(u,y_u(\omega),\omega))^* p = \rho'(u)-\cB^*(\omega) p
\end{equation*}
for $p$ satisfying $\eqref{eq:adjoint_p}.$
We have
\begin{equation*}
\langle \rho'(u),v \rangle_{U^*,U}=(\nabla \rho(u),v )_{U}, \quad
\langle \cB^*(\omega)p ,v \rangle_{U^*,U}= (\cB^*(\omega)p ,v )_{U}, \label{eq:stoch_grad2}
\end{equation*}
from which we obtain \eqref{eq:stoch-gradient}.
To prove \eqref{eq:estimate_dual_arc_1}, we note that $y_u$ satisfies \eqref{eq:PDE_constraint}. Using \eqref{coercivity} and $\cN(0,\omega) = 0$ yields
\begin{align*}
&C(\omega) \lVert y_u(\omega) \rVert_{Y}^{1+\gamma} + \langle \cN(y_u(\omega), \omega) - \cN(0,\omega), y_u(\omega) \rangle_{Y^*,Y} \\
&\quad \leq \langle \cb(\omega) +  \cB(\omega) u , y_u(\omega)\rangle_{Y^*,Y}\\
&\quad \leq ( \lVert \cb(\omega)\rVert_{Y^*}+\lVert \cB(\omega) u \rVert_{Y^*} )\lVert y_u(\omega)\rVert_{Y}. 
\end{align*}
Using the monotonicity of $\cN$
and making elementary manipulations we arrive at  \eqref{eq:estimate_dual_arc_1}. The estimate \eqref{eq:estimate_dual_arc} is derived in a similar way.

To show \eqref{eq:Lipschitz_y}, note that
\begin{equation*}
    \begin{aligned}
    &C(\omega) \lVert y_{u_1}(\omega)-y_{u_2}(\omega) \rVert_{Y}^{1+\gamma} + \langle \cN(y_{u_1}(\omega), \omega) - \cN(y_{u_2}(\omega),\omega), y_{u_1}(\omega)-y_{u_1}(\omega) \rangle_{Y^*,Y} \\
&\quad \leq \langle \cB(\omega) (u_1 -u_2), y_{u_1}(\omega)-y_{u_2}(\omega)\rangle_{Y^*,Y}\\
&\quad \leq \lVert \cB(\omega) (u_1-u_2) \rVert_{Y^*} \lVert y_{u_1}(\omega)-y_{u_2}(\omega)\rVert_{Y},
    \end{aligned}
\end{equation*}
where we used the monotonicity of $\mathcal{N}$ with respect to the first argument. Routine manipulations yield \eqref{eq:Lipschitz_y}. 
\end{proof}

The next result gives sufficient conditions for the continuous Fr\'echet differentiability of the objective and justifies calling the function $G$ a stochastic gradient.
\begin{proposition}\label{prop:differentiability_j}
Let Assumptions \ref{hp1}--\ref{objective-function} hold with $\gamma=1$. Then, 
$j\colon U\rightarrow\mathbb{R}$ is continuously Fr\'echet differentiable.

Assume, moreover, that $u \mapsto \| \rho'(u)\|_{U^*}$ is bounded on bounded sets and that for some $\alpha >0$ we have
\begin{equation}
\label{eq:growth-condition}
\| D_y \tilde{J}(y_u(\omega),\omega)\|_{Y^*}\leq c_1(\omega) \| u \|^\alpha_{U}+\tilde{c}_1(\omega) \textit{ a.s.},
\end{equation}
where
the random variables $$c_2(\omega):= \| \mathcal{B}^* (\omega)\|_{\mathcal{L}(Y,U^*)}  C^{-1}(\omega) c_1(\omega) \text{ and }  \tilde{c}_2(\omega):= \| \mathcal{B}^* (\omega)\|_{\mathcal{L}(Y,U^*)}  C^{-1}(\omega) \tilde{c}_1(\omega)$$
belong to $L^1(\Omega)$. Then $\nabla j(u)=\mathbb{E}[G(u,\cdot)]$  for all $u \in U$.  
\end{proposition}
\begin{proof}
The continuous Fr\'echet differentiability of $j$ follows from that of $\rho$ and $F$ from Assumption~\ref{objective-function} combined with the continuous Fr\'echet differentiability of $u\mapsto y_u$ afforded by \cite[Proposition 3.9]{KS}.

To prove that the equality $\mathbb{E}[G(u,\cdot)]=\nabla j(u)$ holds true, we verify the conditions of Lemma~\ref{lemma:frechet-exchange-derivative-expectation}. That $j$ is well-defined and finite-valued follows from Assumption~\ref{objective-function}. That $J(\cdot,\omega)$ is a.s.~Fr\'echet differentiable follows from the following observation: $S(\omega)\colon U \rightarrow Y$ can be shown to be a.s.~(continuously) Fr\'echet differentiable by invoking the Implicit Function Theorem following the arguments in \cite[Proposition 3.9]{KS} for fixed $\omega$. Then the Fr\'echet differentiability of $J(\cdot,\omega)$ follows using the assumptions on $\tilde{J}$ and $\rho$ from Assumption~\ref{assumption_diff}.
Now, fix an arbitrary $u\in U$ and some bounded neighborhood $V_u\subset U$ of $u$. We observe that for a.e.~$\omega \in \Omega$ and all $v\in V_u$, by \eqref{eq:estimate_dual_arc},
\begin{align*}
 \| G&(v,\omega) \|_{U^*}  \leq \|  \rho'(v) \|_{U^*} +  \| \cB^*(\omega)\|_{\mathcal{L}(Y, U^*)} \| p_u(\omega) \|_{Y} \\
 &\leq  \|  \rho'(v) \|_{U^*} +  \| \cB^*(\omega)\|_{\mathcal{L}(Y, U^*)} C^{-1}(\omega) \| D_y \tilde{J}(y_v(\omega),\omega)  \|_{Y^*} \\
 &\leq \sup_{v\in V_u} \|  \rho'(v) \|_{U^*} +  \| \cB^*(\omega)\|_{\mathcal{L}(Y, U^*)}  C^{-1}(\omega)( c_1(\omega)\sup_{v\in V_u}\|v\|_{U}^\alpha + \tilde{c}_1(\omega))
 =: C_J(\omega).
\end{align*}
The random variable $C_J$ is integrable by the assumptions of the proposition. 
This argument can be repeated for any $u \in U$;  Lemma~\ref{lemma:frechet-exchange-derivative-expectation} yields that $\mathbb{E}[G(u,\cdot)]= \nabla j(u)$ for all $u$.  
\end{proof}
\begin{remark}
\label{rem-growth-condition}
The condition \eqref{eq:growth-condition} is reasonable if one considers that $D_y \tilde{J}(y(\cdot),\cdot)$ belonging to $L^{\frac{pq}{q-p}}(\Omega,Y)$ as required by Assumption~\ref{objective-function} can be guaranteed by the growth condition $$\| D_y \tilde{J}(y,\omega)\|_{Y^*} \leq d_1(\omega)+d_2 \| y\|_Y^{q/\tilde{p}}$$ for $\tilde{p}:=\tfrac{pq}{q-p}$ and $d_1 \in L^{\tilde{p}}(\Omega)$, $d_2 \geq 0.$ In combination with the estimate \eqref{eq:estimate_dual_arc_1}, one gets a condition of the form \eqref{eq:growth-condition}.
Integrability of $c_2$ and $\tilde{c}_2$ can be verified in applications using H\"older inequalities. 
\end{remark}
To verify the Lipschitz gradient condition required by Assumption~\ref{asu1ii}, we will need the following notion: given $\Phi\colon Y\times \Omega \rightarrow Y^*$, we say that $\Phi(\cdot,\omega)$ is locally Lipschitz for almost all $\omega \in \Omega$ if 
for any $M>0$ there exists a constant $L(M,\omega)>0$ such that
\begin{equation}
\label{eq:local-lipschitz}
    \| \Phi(y,\omega)-\Phi(z,\omega) \|_{Y^*} \leq L(M,\omega)\| y-z\|_{Y} \quad \text{a.s.}
\end{equation}
for all $y,z \in L^\infty(D)$ such that $\|y\|_{L^\infty(D)}\leq M$ and $\|z \|_{L^\infty(D)} \leq M.$ Many examples of interest satisfy this assumption as demonstrated in \cite[Lemma 4.11]{tro}; recall that $H^1(D)=Y \hookrightarrow L^2(D) = L^2(D)^* \hookrightarrow Y^*=H^{-1}(D)$, so a condition of the form
\begin{equation*}
    \| \Phi(y,\omega)-\Phi(z,\omega) \|_{L^2(D)} \leq L(M,\omega) \| y-z\|_{L^2(D)}
\end{equation*}
as in \cite[Lemma 4.11]{tro} implies \eqref{eq:local-lipschitz}.
\begin{lemma} \label{lemma:basic_estimates}
Let Assumptions \ref{hp1}--\ref{objective-function} hold. 
Assume that for almost all $\omega\in \Omega$, $y \mapsto D_y \tilde{J}(y,\omega)$ and $y\mapsto \mathcal{N}^*_y(y,\omega)$ are locally Lipschitz  \footnote{For $\mathcal{N}^*_y$, we mean the following: for all $M>0$, we have the existence of a constant $L_{\mathcal{N}^*_y}(M,\omega)$ such that $\|  \mathcal{N}^*_y(y_2,\omega) - \mathcal{N}^*_y(y_1,\omega) \|_{\mathcal{L}(Y,Y^*)} \leq L_{\mathcal{N}^*_y}(M,\omega) \| y_2-y_1 \|_Y$ for almost all $\omega \in \Omega$ and for all $y_i$ such that $\| y_i\|_{L^\infty(D)} \leq M$, $i=1, 2$.} mappings. 
Then, given $M>0$ and $\omega$, we have the almost sure bound 
\begin{align}\label{eq:Lipschitz_p}
   \| p_{u_2}(\omega) - p_{u_1}(\omega) \|_{Y}^{\gamma} \leq L(M,\omega)\| y_{u_2}(\omega) - y_{u_1}(\omega) \|_{Y} 
\end{align}
for all $u_i \in U$ with states $y_{u_i}(\omega) \in L^\infty(D)$ satisfying $\| y_{u_i}(\omega)\|_{L^\infty(D)}\leq M$ for $i=1,2$, where
\begin{equation}\label{eq:cont_L}
L(M,\omega):=\left( L_{\tilde{J}_y}(M,\omega) +L_{\mathcal{N}^*_y}(M,\omega)C^{-1}(\omega) \| D_y \tilde{J}(y_{u_1}(\omega),\omega)\|_{Y^*}  \right) C^{-1}(\omega) 
\end{equation}
and $L_{\tilde{J}_y}(M,\omega)$ and $L_{\mathcal{N}^*_y}(M,\omega)$ are the Lipschitz constants for $\tilde{J}_y(\cdot,\omega)$ and $\mathcal{N}^*_y(\cdot,\omega)$, respectively.
\end{lemma}
\begin{proof}
The adjoint equation yields that
\begin{align*}
  &\langle ( \mathcal{A}^*(\omega)+  \mathcal{N}^*_y(y_{u_2}(\omega),\omega))p_{u_2}(\omega)\\
  & \quad- ( \mathcal{A}^*(\omega)+  \mathcal{N}^*_y(y_{u_1}(\omega),\omega))p_{u_1}(\omega),p_{u_2}(\omega)-p_{u_1}(\omega) \rangle_{Y^*,Y}    \\
   & = - \langle D_y \tilde{J}(y_{u_2}(\omega),\omega)- D_y \tilde{J}(y_{u_1}(\omega),\omega) , p_{u_2}(\omega)-p_{u_1}(\omega) \rangle_{Y^*,Y}  .
\end{align*}
Thus,
\begin{align*}
   &C(\omega)\| p_{u_2}(\omega)-p_{u_1}(\omega)\|_Y^{\gamma+1}\\
   & \quad +\langle  \mathcal{N}^*_y(y_{u_2}(\omega),\omega)(p_{u_2}(\omega)-p_{u_1}(\omega)) , p_{u_2}(\omega)-p_{u_1}(\omega) \rangle_{Y^*,Y}\\
   &\quad + \langle    (\mathcal{N}^*_y(y_{u_2}(\omega),\omega)-\mathcal{N}^*_y(y_{u_1}(\omega),\omega)) p_{u_1}(\omega),p_{u_2}(\omega)-p_{u_1}(\omega) \rangle_{Y^*,Y} \\
   &\leq - \langle D_y \tilde{J}(y_{u_2}(\omega),\omega)- D_y \tilde{J}(y_{u_1}(\omega),\omega) , p_{u_2}(\omega)-p_{u_1}(\omega) \rangle_{Y^*,Y},
\end{align*}
where we used Assumption \ref{hp1}. By the nonnegativity of the operator $\mathcal{N}_y$,
\begin{align*}
   &C(\omega)\| p_{u_2}(\omega)-p_{u_1}(\omega)\|_Y^{\gamma+1}\\
 &\quad + \langle    (\mathcal{N}^*_y(y_{u_2}(\omega),\omega)-\mathcal{N}^*_y(y_{u_1}(\omega),\omega)) p_{u_1}(\omega),p_{u_2}(\omega)-p_{u_1}(\omega) \rangle_{Y^*,Y} \\
   &\leq - \langle D_y \tilde{J}(y_{u_2}(\omega),\omega)- D_y \tilde{J}(y_{u_1}(\omega),\omega) , p_{u_2}(\omega)-p_{u_1}(\omega) \rangle_{Y^*,Y}.
\end{align*}
One can conclude by the local Lipschitz continuity of the maps $y \mapsto D_y \tilde{J}(y,\omega)$ and $y\mapsto \mathcal{N}^*_y(y,\omega)$ as well as \eqref{eq:estimate_dual_arc}. 
\end{proof}
In Section \ref{sec:numerics} below, we provide a model problem together with a discussion of the Lipschitz conditions required in the previous lemma. Among the other things, we show that in this example the Lipschitz condition for $\mathcal{N}_y^*(\cdot,\omega)$ can be deduced by using the following property: the states $y_u(\omega)$ are essentially bounded, namely, for almost all $\omega \in \Omega$, $\|y_{u}(\omega)\|_{L^{\infty}(D)} \leq c_{\infty}(\omega) \| u\|_{U}$ for some $c_{\infty}(\omega)$ independent of $u$. 

The following lemma proves that, under suitable assumptions, the conditions required for measurability as stated in Remark~\ref{rem:measurability} are satisfied. Moreover, we verify also Assumption \ref{asu1ii}  and \ref{assumption:well-posedness-unconstrained}.
\begin{proposition}\label{prop:2.2verified}
Let the assumptions of Lemma \ref{lemma:basic_estimates} hold and
assume $\gamma =1$. Moreover, suppose that for almost all $\omega \in \Omega$, 
\begin{equation}\label{eq:casas_L_infty_estim}
\|y_{u}(\omega)\|_{L^{\infty}(D)} \leq c_{\infty}(\omega) \| u\|_{U}
\end{equation}
for some random variable $c_{\infty}\colon \Omega \rightarrow [0,\infty)$ independent of $u$.
\begin{enumerate}
\item Then, the stochastic gradient defined by \eqref{eq:stoch-gradient} fulfills the Carath\'eodory property. 
\end{enumerate}
In addition to the above assumptions, 
let the assumptions in Proposition \ref{prop:differentiability_j} hold.  
\begin{enumerate}
\item[\textup{2.}]  
For $V\subset H$ bounded, set $\beta:= \sup_{v\in V}\| v\|_U$.  Suppose that $\nabla \rho$ is Lipschitz continuous on $V$ with constant $L_{\rho'}(V).$ Assume that the random variable 
$$c_3(\omega):=\| \cB(\omega)\|^{2}_{\mathcal{L}(U,Y^*)} L( \beta c_{\infty}(\omega) ,\omega) C(\omega)^{-1}$$
belongs to $L^1(\Omega)$. Then, the gradient $\nabla j$ is 
Lipschitz continuous on $V$ with  $L_{j'}(V)=L_{\rho'}(V)  + \mathbb{E}[ c_3(\cdot)]$.
\item[\textup{3.}] Suppose that
$c_2$ and $\tilde{c}_2$ from Proposition \ref{prop:differentiability_j} belong to $L^{2}(\Omega)$.  Then, the stochastic gradient defined by \eqref{eq:stoch-gradient} fulfills Assumption~\ref{assumption:well-posedness-unconstrained}.
\end{enumerate}
\end{proposition}
\begin{proof}
First, we prove the continuity of $u \mapsto G(u,\omega)$ on $U$ for almost all $\omega \in \Omega$.
To that end, we fix $\omega$, $u_1$, and $\delta >0$ and construct a bound $M=M(\omega)$ so that $\|y_{u_2}\|_{L^\infty(D)}\leq M$ for every $u_2$ in a $\delta$-neighborhood of $u_1.$
By assumption \eqref{eq:casas_L_infty_estim}, it holds $$\| y_{u_2}(\omega)\|_{L^{\infty}(D)}\leq c_{\infty}(\omega)\| u_2\|_{U} \leq c_{\infty}(\omega)( \delta + \| u_1 \|_U):=M=M(\omega)$$ for every $u_2\in U$ such that $\| u_2 - u_1 \|_U \leq \delta$ and for almost all fixed $\omega$. 
Now, we make use of the identity $\|\cB(\omega)\|_{\mathcal{L}(U,Y^*)} = \|\cB^*(\omega)\|_{\mathcal{L}(Y,U^*)}$ and \eqref{eq:Lipschitz_p} followed by \eqref{eq:Lipschitz_y}:
\begin{align}
&\| G(u_1,\omega) - G(u_2,\omega)\|_{U^*} \label{eq:est_G_carat} \\
&\quad\leq \| \rho’(u_1) - \rho’(u_2)\|_{U^*} + \| \cB^*(\omega)(p_{u_1}(\omega)-p_{u_2}(\omega))\|_{U^*}  \nonumber \\
& \quad\leq \|\rho’(u_1) - \rho’(u_2)\|_{U^*} + \lVert \mathcal{B}(\omega)\rVert_{\mathcal{L}(U,Y^*)}L(M,\omega) \|y_{u_1}(\omega)-y_{u_2}(\omega)\|_Y \nonumber \\
& \quad\leq \|\rho’(u_1) - \rho’(u_2)\|_{U^*} + \lVert \mathcal{B}(\omega)\rVert_{\mathcal{L}(U,Y^*)}^2 L(M,\omega) C^{-1}(\omega) \|u_1-u_2\|_{U} \nonumber
\end{align}
for every $u_2$ in a $\delta$-neighborhood of $u_1$.
This yields the continuity of $u\mapsto G(u,\omega)$ on $U$ for almost all fixed $\omega \in \Omega$. 

Now, we use the fact that $G(u,\omega)=\nabla \rho(u) + \mathcal{B}^*(\omega)p_u(\omega)$ together with $p_u \in L^q(\Omega,Y)$ by \cite[Proposition 3.2]{KS} and $\cB \in L^s(\Omega, \mathcal{L}(U,Y^*))$ by Assumption~\ref{assumption_N_Ny}. In particular, $\omega \mapsto p_u(\omega)$ is strongly measurable for every $u$ and $\cB^*$ is a uniformly measurable random operator. Their product is therefore measurable, making $\omega \mapsto G(u,\omega)$ measurable for all $u$. 

Let us prove claim 2. 
Taking $v \in V$, then by assumption, for almost all $\omega\in \Omega$, we have that $\| y_{v}(\omega)\|_{L^{\infty}(D)}\leq c_{\infty}(\omega)  \, \beta$ for all states $y_v(\omega)$ corresponding to a control $v\in V$.
Using calculations similar to \eqref{eq:est_G_carat} and Jensen's inequality, the claim follows from
\begin{equation*}
\label{eq:Lip_const_nablaj}
\begin{aligned}
\| \nabla j(u_1)-\nabla j(u_2) \|_{U} 
&\leq \mathbb{E}[  \| G(u_1,\cdot)-G(u_2,\cdot)\|_{U^*} ] \\
& \leq (L_{\rho'}(V)+ \E[c_3( \cdot)] )\| u_1-u_2\|_U,
\end{aligned}
\end{equation*}
together with the integrability of the random variable $c_3$.
For the third claim, we utilize the estimates \eqref{eq:estimate_dual_arc} and \eqref{eq:growth-condition} to obtain that
\begin{equation}
\label{eq:bound-on-gradient-example}
\begin{aligned}
    \| G(u,\omega)\|_{U^*} &\leq \| \rho'(u) \|_{U^*} +  \| \cB(\omega) \|_{\mathcal{L}(U,Y^*)} \| p_u(\omega) \|_{U}\\
    &\leq \| \rho'(u) \|_{U^*} +  \| \cB(\omega) \|_{\mathcal{L}(U,Y^*)} C^{-1}(\omega) \| D_y \tilde{J}(y_u(\omega),\omega)\|_{Y^*} \\
     &\leq \| \rho'(u) \|_{U^*} +  \| \cB(\omega) \|_{\mathcal{L}(U,Y^*)} C^{-1}(\omega) (c_1(\omega) \| u \|^\alpha_{U}+\tilde{c}_1(\omega) )\\
     &= \| \rho'(u) \|_{U^*} + c_2(\omega) \|u\|_U^\alpha + \tilde{c}_2(\omega).
\end{aligned}
\end{equation}
Therefore
\begin{align*}
\mathbb{E}[ \| G(u,\cdot)\|_{U^*}^2] &\leq 2 \| \rho'(u) \|_{U^*}^2 + 2\E[ (c_2\|u\|_U^{\alpha} + \tilde{c}_2)^2] \\
&\leq 2 \| \rho'(u) \|_{U^*}^2 + 4\E[ c_2^2]\|u\|_U^{2\alpha}+4\E[\tilde{c}_2^2] =:M(u).
\end{align*}

Using the square integrability of $c_2$ and $\tilde{c}_2$ as well as the local boundedness property for $\rho'$ from Proposition \ref{prop:differentiability_j}, it follows that $u \mapsto M(u)$ is bounded on bounded sets as required.
\end{proof}
\begin{remark}
\label{rem-Prop35}
Here are some comments on the above proposition.  First, using Lemma~\ref{lem:sufficient-conditions-boundedness}, it is possible to prove that the sequence of iterates produced by Algorithm~\ref{alg:PSG_Hilbert_Nonconvex_Unconstrained} are almost surely contained in a bounded set when considering the tracking-type functional and Tikhonov regularization in \eqref{eq:example_problem} with $\lambda$ sufficiently large. 
For the growth conditions on the second to fourth moments of the stochastic gradient, we see that
\begin{align*}
\lVert G(u,\omega)\rVert_U^k &\leq  (\lambda \lVert u \rVert_U + c_2(\omega) \lVert u \rVert_{U} + \tilde{c}_2(\omega))^k\\
& \leq ((\lambda + c_2(\omega) )\lVert u \rVert_{U} + \tilde{c}_2(\omega))^k \leq 2^{k-1}( (\lambda + c_2(\omega) )^k\lVert u \rVert_{U}^k + \tilde{c}_2(\omega)^k) .
\end{align*}
For $\lVert u \rVert_U \geq 1$, we have
\begin{align*}
(\nabla j(u),u)_U &= (\E[G(u,\cdot)],u)_U = (\lambda u- \E[\mathcal{B}^* p_u],u)_U\\
& \geq \lambda \lVert u \rVert_U^2  -  \E[c_2]  \lVert u\rVert_U^2 - \E[\tilde{c}_2] \lVert u\rVert_U\\
& \geq \lambda \lVert u \rVert_U^2 - \hat{c}_2 \lVert u\rVert_U^2
\end{align*}
for some $\hat{c}_2>0.$ 
Condition \eqref{eq:angle-condition} is fullfilled if $\lambda -\hat{c}_2\geq0$
and this certainly is the case for $\lambda$ sufficiently large. On the other hand, we observe that $\hat{c}_2$ comes from rather coarse estimates. Moreover, numerical experiments suggest that the above condition for $\lambda$ is not the optimal one (see Section \ref{sec:numerics}). One may naturally check the boundedness of $u_n$ numerically, but further sufficient conditions are needed, since the literature provided so far requires assumptions that appear to be unnecessarily  strong.



Next, in Assumption~\ref{asu1ii} we required that $j\in C_L^{1,1}(V)$ for some open set $V$ with $U \subset V$, where $U$ is a bounded set of iterates. Thus, we can restrict $V$ to be bounded in claim 2. Second, in several applications (as in, e.g.,~Section \ref{sec:numerics}) the term $\mathcal{N}_y$ might be independent of $\omega$. If one can additionally prove that the states $y_u$ belong to $L^{\infty}(\Omega, L^{\infty}(D))$, which would follow from an estimate of the form \eqref{eq:estimate_dual_arc_1} along with suitable integrability assumptions on $C^{-1}$, $\mathfrak{b}$, and $\mathcal{B}$, then it is then clear that the term $L$ appearing in the definition of $c_3$ is deterministic. This simplifies the study of the integrability of the random variable $c_3$. However, proving that $y_u(\omega)$ belongs to $L^{\infty}(D)$ for almost all $\omega \in \Omega$ and the corresponding estimate \eqref{eq:casas_L_infty_estim} might require in general delicate analysis, as observed in \cite[Section 4.2.3]{tro} for the case of deterministic semilinear problems. 

Next, in Assumption~\ref{asu1ii} we required that $j\in C_L^{1,1}(V)$ for some open set $V$ with $U \subset V$, where $U$ is a bounded set of iterates. Thus, we can restrict $V$ to be bounded in claim 2. Second, in several applications (as in, e.g.,~Section \ref{sec:numerics}) the term $\mathcal{N}_y$ might be independent of $\omega$. If one can additionally prove that the states $y_u$ belong to $L^{\infty}(\Omega, L^{\infty}(D))$, which would follow from an estimate of the form \eqref{eq:estimate_dual_arc_1} along with suitable integrability assumptions on $C^{-1}$, $\mathfrak{b}$, and $\mathcal{B}$, then it is then clear that the term $L$ appearing in the definition of $c_3$ is deterministic. This simplifies the study of the integrability of the random variable $c_3$. However, proving that $y_u(\omega)$ belongs to $L^{\infty}(D)$ for almost all $\omega \in \Omega$ and the corresponding estimate \eqref{eq:casas_L_infty_estim} might require in general delicate analysis, as observed in \cite[Section 4.2.3]{tro} for the case of deterministic semilinear problems. 
\end{remark}

Finally, we discuss the assumptions needed for the convergence of limit points of Algorithm \ref{alg:PSG_Hilbert_Nonconvex_Unconstrained}, namely, for the claims 1 and 2 of Theorem \ref{theorem:convergence-algorithm1}. 
\begin{proposition}\label{proposition_wls_PDEconstraint}
Let all the assumptions in Proposition \ref{proposition_conditions_limit_points} hold true.
If  $u \mapsto \nabla \rho(u)$ is weakly sequentially continuous,  then $u \mapsto \| \nabla j(u) \|_{U}$ is weakly lower semicontinuous.
\end{proposition}
\begin{proof}
We need to verify that
\begin{equation*}\label{eq:wls_nablaj}
u_k \rightharpoonup \bar u  \mbox{ in }U \quad \Rightarrow \quad  \liminf_{k\rightarrow  \infty} \| \nabla j(u_k) \|_{U}\geq \| \nabla j(\bar{u}) \|_{U}.
\end{equation*}
Thanks to Proposition \ref{proposition_conditions_limit_points}, we have $\nabla j(u)= \nabla \rho(u)- \mathbb{E}[ \mathcal{B}^* p_u]$. Furthermore, the mapping $u \mapsto p_u$ is completely continuous from $U$ to $L^q(\Omega,Y)$ as stated in \cite[Proposition 3.3]{KS}. Hence $u_k \rightharpoonup \bar{u}$ in $U$ implies $p_{u_k} \rightarrow p_{\bar{u}}$ in $L^q(\Omega,Y)$. Moreover, $\cB^*$ as a map from $L^q(\Omega,Y)$ to $L^p(\Omega, U^*)$ is bounded according to  \cite[Lemma 2.1]{KS}. In particular, $ \| \cB^* p_{u_k}- \cB^* p_{\bar{u}}\|_{L^p(\Omega,U^*)}\rightarrow 0$. 
    This implies that $\mathbb{E}[\cB^* p_{u_k}] \rightarrow \mathbb{E}[\cB^* p_{\bar{u}}] $ as $k$ goes to infinity.
Recall that $\nabla \rho$ is weakly sequentially continuous.
Then, weak lower semicontinuity of the norm, i.e., 
$$ \eta_n \rightharpoonup  \eta \mbox{ in } U \quad \Rightarrow \quad \|\eta \|_{U} \leq  \liminf_{n \rightarrow \infty} \| \eta_n \|_{U}$$
can be applied to the sequence $\eta_n= \nabla \rho(u_n)- \mathbb{E}[ \cB^* p_{u_n}]$ and the weak limit point $\bar{\eta}=\nabla \rho(\bar{u})- \mathbb{E}[( \cB^* p_{\bar{u}})]$ to yield the conclusion of the proposition.
\end{proof}
Let us briefly comment on Proposition \ref{proposition_wls_PDEconstraint}.
Assuming weak sequential continuity of $ u \mapsto  \nabla \rho(u)$ in $U$ is rather harmless and it is trivially satisfied, for instance, in the case where $\rho$ is a regularizing term as in the basic problem \eqref{eq:example_problem}. 
\begin{remark}
\label{rem-bounded-iterates}
In this paper, the control set is unbounded and we still can guarantee the existence of an optimal control when $\lambda>0$.  Indeed, taking $\lambda>0$  in \eqref{eq:example_problem}, $j$ is (radially) coercive in the sense that
\begin{equation}\label{eq:corcivity_lambda}
    j(u)\geq \frac{\lambda}{2} \| u\|_{L^2(D)}^2 \rightarrow \infty \quad \text{ as } \quad \|u \|_{L^2(D)}\rightarrow \infty.
\end{equation}
Thus, a minimizing sequence is bounded and, in turn, an optimal solution exists by the weakly lower semicontinuity of the map $u \mapsto \rho(u)$ and the weak continuity of the map $u\mapsto \mathbb{E}[\tilde{J}(S(\cdot)u, \cdot)]$ (see e.g. \cite[Prop. 4.2]{KS} and reference therein).
\end{remark}

\subsection{Second-order derivative of $J(u,\omega)$}
\label{sec:second-order}
In Theorem \ref{theorem:convergence-algorithm1} and in Lemma \ref{lemma:twice-differentiability}, we have seen that the twice Fr\'echet differentiability of the reduced cost function $j$ is part of the sufficient conditions guaranteeing the convergence of weak limit points from Algorithm \ref{alg:PSG_Hilbert_Nonconvex_Unconstrained}. 
For this reason, in this subsection we aim to investigate the twice Fr\'echet differentiability of the mapping $j$. For the purposes of this paper, we will keep the discussion formal. We start by providing an expression for the second derivative of $J(u,\omega)$  (see also \cite[p.~64]{Hinze} and \cite[ch.~4]{tro}, respectively, for a similar discussion in a deterministic setting). 

Hereafter, let Assumptions \ref{hp1}--\ref{objective-function} hold true. Moreover, assume that $\tilde{J}_\omega:=\tilde{J}(\cdot,\omega)$, $e(\cdot,\cdot,\omega)$ and $\rho$ are twice continuously Fr\'echet differentiable in $Y$, $ U\times Y$ and $U$, respectively.
We use the notation $S_{\omega}u:= S(\omega)u$ and throughout, statements are assumed to hold for almost every $\omega$. 
For $u\in U$ and a direction $s_1\in U$, the chain rule yields that 
\begin{equation*}\label{eq:first_derivative_sensitivity}
    \langle J'_{\omega}(u),s_1 \rangle_{U^*,U} = \langle \rho'(u), s_1 \rangle_{U^*,U} + \langle \tilde{J}'_{\omega}(S_{\omega}u),S'_{\omega}us_1 \rangle_{Y^*,Y}.
\end{equation*} Differentiating $\langle J'_{\omega}(u),s_1 \rangle_{U^*,U}$ in the direction $s_2\in U$, 
\begin{align*}
    \langle J''_{\omega}(u) s_2,s_1 \rangle_{U^*,U} &= \langle \rho''(u)s_2,s_1 \rangle_{U^*,U} + \langle \tilde{J}''_{\omega}(S_{\omega}u) S'_{\omega}us_2, S'_{\omega}us_1 \rangle_{Y^*,Y} \nonumber \\
    &\quad + \langle \tilde{J}'_{\omega}(S_{\omega}u), S''_{\omega}(u)(s_1,s_2)\rangle_{Y^*,Y}.
\end{align*}
This shows that
\begin{equation}
\label{eq:second-deriv}
J''_{\omega}(u)=\rho''(u)+ (S''_{\omega}u)^* \tilde{J}'_{\omega}(S_{\omega}u)+ (S'_{\omega}u)^*  \tilde{J}''_{\omega}(S_{\omega}u)S'_{\omega}u.    
\end{equation}

Therefore, to introduce the second-order derivative $J''_{\omega}(u)$ the following observations are required:
\begin{enumerate}
    \item 
First, we establish that there exists a unique solution $v(\omega)\in Y$ to the equation
  \begin{equation}\label{eq:sensitivity_dery}
        (\cA(\omega)+\cN_y (S_{\omega}u,\omega))v(\omega) = w 
    \end{equation}
for any $w \in Y^*.$ The operator $\cA(\omega)+\cN_y(y,\omega)$ is surjective from $Y$ into $Y^*$ since $\mathcal{N}_y(y,\omega)$ is maximally monotone as argued in the proof of \cite[Proposition 3.9]{KS}. This in combination with strong monotonicity gives unique solvability of \eqref{eq:sensitivity_dery}. The Implicit Function Theorem gives that $S_\omega$ is continuously Fr\'echet differentiable, where for a fixed $s \in U$, $v(\omega)=S'_\omega u s \in Y$ solves the equation
    %
    \begin{equation}\label{eq:sensitivity_dery2}
        (\cA(\omega)+\cN_y (S_{\omega}u,\omega))v(\omega)= \cB(\omega)s.
    \end{equation}
    or equivalently
    \begin{equation*}
    D_y e(u,S_{\omega}u,\omega)v(\omega)+D_u e(u,S_{\omega}u,\omega)=0.
    \end{equation*}
In fact, $v \in L^q(\Omega,Y)$ using identical arguments to those used in \cite[Theorem 3.5]{KS}.
\item One can compute $(S'_{\omega}u)^*h\in U$ via 
\begin{equation*}
    (S'_{\omega}u)^*h= -D_u e(u,S_{\omega}u,\omega)^* D_y e (u,S_{\omega}u,\omega)^{-*} h, 
\end{equation*}
or, equivalently, by an adjoint equation formulation,
\begin{equation}\label{eq:adj_du}
   (S'_{\omega}u)^*h= \cB^*(\omega) h_1, 
\end{equation}
where $h_1=h_1(u,\omega)$ solves
\begin{equation}\label{eq:sensitivity_adjoint}
(\cA(\omega)+\cN_y(S_{\omega}u,\omega))^* h_1 = h.
\end{equation}
The existence of a unique solution to \eqref{eq:sensitivity_adjoint} can be argued as in the previous point but reasoning with the adjoint operator $(\cA(\omega)+\cN_y(S_{\omega}u,\omega))^*$. The integrability of the solutions of \eqref{eq:adj_du}-\eqref{eq:sensitivity_adjoint} can be investigated as in \cite[Proposition 4.3]{KS}. The difference from \cite[Proposition 4.3]{KS} lies in the fact that the regularity of $h_1$ depends on that of $h=\tilde{J}''_\omega(S_\omega u)S'_\omega u$.  
\item \emph{Second derivative of the control-to-state mapping.} 
To show the existence of second derivative for the control-to-state mapping, we apply the Implicit Function Theorem. 
Taking $y=\mathcal{S}_{\omega}u$ in \eqref{eq:PDE_constraint} and differentiating this expression in the direction $s_1\in U$,
\begin{equation*}
\mathcal{A}(\omega)S'_{\omega}us_1+\mathcal{N}_y (S_{\omega}u,\omega)S'_{\omega}us_1 - \mathcal{B}(\omega)s_1=0.
\end{equation*}
We calculate next the directional derivative in the direction $s_2\in U$,
\begin{equation}\label{eq:second_sensitivity}
 (\mathcal{A}(\omega)+\mathcal{N}_y(S_{\omega}u,\omega))S_{\omega}''u (s_1,s_2) + \mathcal{N}_{yy}(S_{\omega}u,\omega)(S'_{\omega}u s_1,S'_{\omega}u s_2)=0.   
\end{equation}
Here, $S''_{\omega}(u)\in \mathcal{L}(U,\mathcal{L}(U,Y))$ and  $(S'_{\omega}(u)s_1)'s_2=S''_{\omega}(u)(s_1,s_2).$ The unique solvability of \eqref{eq:second_sensitivity} is given again by the fact that the operator $ D_y  e(u,y,\omega)=(\cA(\omega)+\cN_y(y,\omega))$ is surjective from $Y$ to $Y^*$ and strongly monotone. Summarizing, the second derivative $S''_{\omega}(u)$  is given by $z(\omega)=S''_{\omega}(u) (s_1,s_2),$ where $z(\omega)$ is the unique weak solution to  \begin{equation}\label{eq:second_sensitivity2}
(\cA(\omega)+\cN_y(S_{\omega}u,\omega)) z(\omega) =- \cN_{yy} (S_{\omega}u,\omega)( \psi_1, \psi_2),
\end{equation}
and $\psi_i=\psi_i(\omega)= S'_{\omega}us_i \in Y$ for $i=1,2$. Applying the Implicit Function Theorem, which is possible due to the twice Fr\'echet differentiability of $e(\cdot,\cdot,\omega)$, we obtain that $S$ is twice continuously differentiable.
We investigate here the integrability of $S_{\omega}''u$. 
Let $\bar{q} \in [1,\infty)$ such that $\frac{t\gamma}{t+1} \leq \bar{q} \leq t \gamma$ in case $t<\infty$ ($\bar{q}$ depending on $\gamma$ and $t$ appearing in Assumption \ref{hp1} and \ref{assumption_diff}).
By \eqref{eq:second_sensitivity2}, the fact that $\mathcal{N}_y(\cdot,\omega)$ is a nonnegative linear operator as stated in Assumption \ref{assumption_N_Ny-1} and then \eqref{coercivity}, for a.e.~$\omega\in \Omega$, 
\begin{align*}
    C(\omega)\| z(\omega)\|_Y^{\gamma+1} 
    &\leq  \langle \mathcal{N}_{yy} (S_{\omega}u,\omega)(\psi_1,\psi_2), z(\omega)\rangle_{Y^*,Y}  \\
        &\leq \| \mathcal{N}_{yy} (S_{\omega}u,\omega)  ( \psi_1 , \psi_2)\|_{Y^*} \|z(\omega)\|_Y.
\end{align*}
Thus, we have that
\begin{equation*}
  \| z(\omega)\|^{\bar q}_Y \leq C^{-\frac{ \bar{q}}{\gamma}}(\omega)  \| \mathcal{N}_{yy} (S_{\omega}u,\omega) ( \psi_1, \psi_2) \|_{Y^*}^{\frac{\bar{q}}{\gamma}} \quad\text{ a.s.}
\end{equation*}
Taking the expectation, using the H\"older inequality with $a= \frac{t \gamma}{\bar{q}}\in [1,  \infty)$ (because $\bar{q} \leq t \gamma)$ $b=\frac{t \gamma}{t \gamma - \bar{q}}\in [1,\infty)$, we obtain that 
\begin{equation*}
 \mathbb{E}\left[ \| z(\cdot)\|^{\bar q}_Y \right] \leq \mathbb{E}\left[ C^{-t}\right]^{\frac{\bar{q}}{t \gamma}} \mathbb{E} \left[  \| \mathcal{N}_{yy} (S(\cdot)u,\cdot) ( \psi_1, \psi_2) \|_{Y^*}^{\frac{\bar{q}t}{t \gamma-\bar{q}}} \right]^{\frac{t \gamma-\bar{q}}{t \gamma}}.
\end{equation*}
Therefore, since $C^{-1}\in L^t(\Omega)$ by Assumption \ref{assumption_diff}, the following fact holds true: if $\mathcal{N}_{yy} (S(\cdot)u,\cdot) ( \psi_1, \psi_2)  \in L^{\frac{\bar{q}t}{t\gamma-\bar{q}}}(\Omega,Y^*)$, then $z \in L^{\bar{q}}(\Omega,Y)$. (To show measurability of $z$, one can, for instance, use the assumed measurability of the other operators and apply Filippov's Theorem; see \cite[Theorem 3.5]{KS} for this type of argument.) The condition $ \frac{t \gamma}{t+1}\leq \bar q$ ensures that $\frac{\bar{q}t}{t\gamma-\bar{q}} \geq 1$. The Bochner integrability of $\mathcal{N}_{yy} (S(\cdot)u,\cdot) ( \psi_1, \psi_2)$ depends on that of $S(\cdot)u$, $\psi_1$, and $\psi_2$. With similar calculations one can investigate the case $t=\infty$. Indeed, in the case where $\gamma=1$,
\begin{equation*}
    \| z(\omega)\|_Y \leq C^{-1}(\omega)\| \mathcal{N}_{yy}(S_{\omega}u,\omega)(\psi_1,\psi_2)\|_{Y^{*}}\quad \text{a.s.}
\end{equation*}
Therefore, by H\"older inequality it follows that $z\in L^{1}(\Omega, Y)$ as long as $\mathcal{N}_{yy}(S(\cdot)u,\cdot)\in L^{1}(\Omega,Y^*)$. 
\end{enumerate}
We have thus obtained \eqref{eq:second-deriv} together with the equations solved by  $S'_{\omega}u$ and $S''_{\omega}u$.
The twice Fr\'echet differentiability of $j$ can be investigated remembering the fact that $j(u):=\mathbb{E}[J(u,\cdot)]$ and appealing to Lemma \ref{lemma:frechet-exchange-derivative-expectation} at this point.

We remind the reader that our purpose here is to keep the model rather general, namely, to consider \eqref{eq:main_prb_PDEs}. A complete treatment of the properties of $j''$ can be given by considering a specific application, but this is beyond the scope of this paper.

\section{Numerical Experiments}
\label{sec:numerics}
\paragraph{Test problem}
For the following numerical simulation, we observe a special case of \eqref{eq:PDEs_intro}, namely
\begin{equation}
\label{eq:ex-PDE}
\begin{aligned}
    -\nabla \cdot (a(\cdot,\omega)\nabla y)  + y + y^5 &= u \\
    \frac{\partial y}{\partial n} &= 0,
    \end{aligned}
\end{equation}
where $D=(0,1) \times (0,1).$
The random field $a$ is defined using a truncated Karhunen-Lo\`eve expansion (see \cite[Example 9.37]{lord}) given by
\begin{equation}
\label{eq:random-field-expansion}
a(x,\omega)= \hat{a}(x,\xi(\omega))= a_0+\sum_{i=1}^{20} \sqrt{\eta_i} \phi_i  \xi^{i}(\omega),
\end{equation}
where $a_0=1$ and $\xi^{i}$ is randomly chosen according to the uniform distribution on the interval $[-1,1]$. The eigenfunctions and eigenvalues are given by 
$$\tilde{\phi}_{j,k}(x):= 2\cos(j \pi x_2)\cos(k \pi x_1), \quad \tilde{\eta}_{k,j}:=\frac{1}{4} \exp(-\pi(j^2+k^2)l^2), \quad j,k \geq 1,$$
and are reordered so that the eigenvalues appear in descending order (i.e., $\phi_1 = \tilde{\phi}_{1,1}$ and $\eta_1 = \tilde{\eta}_{1,1}$) and we choose the correlation length $l=0.5$.  We choose an objective of the form \eqref{eq:example_problem} with $y_D(x)=60+160(x_1(x_1-1)+x_2(x_2-1))$.

\paragraph{Verification of problem assumptions} Here, we briefly summarize why this problem fits the assumptions in the previous section. One can verify that $a(\cdot,\omega) \geq 1-\sum_{i=1}^{20} \exp(-\pi(j^2+k^2)l^2) =:C>0$ for the reordered pairs such that $i=1$ corresponds to $(j,k)=(1,1)$ and $i=20$ corresponds to $(j,k)=(4,4)$. Hence, the operator $\mathcal{A}(\omega)$ induced by the bilinear form $\int_D a(x,\omega) \nabla y \cdot \nabla v + yv \D x$ fulfills the coercivity condition \eqref{coercivity} with $\gamma=1$. It is also clearly bounded as an operator from $Y:=H^1(D)$ to $Y^*=H^{-1}(D).$ Note that the embedding $\mathcal{B}$ mapping elements of $L^2(D)$ to themselves in $H^{-1}(D)$ is compact and hence completely continuous. For this example, we have $\mathfrak{b} \equiv 0$ and $\mathcal{N}(y):=y^5$, which is maximally  monotone. Hence, Assumption \ref{hp1} is satisfied. Assumption \ref{assumption_diff} is also satisfied, since all operators are deterministic other than $\mathcal{A}(\cdot)$, and the fact that $\xi^i$ are uniformly distributed in $[-1,1]$ makes $\mathcal{A}(\cdot)y \in L^\infty(\Omega,Y)$ for all $y \in Y$.


The mapping $\mathcal{N}$ is continuously differentiable from $L^{10}(D)$ to $L^2(D)$ with $\mathcal{N}'(y)v= 5y^4v$; one can verify this fact by routine calculations based on the H\"older inequality (see, e.g., \cite[Lemma 1.13]{Hinze}). 
By the Sobolev embedding theorem, for $n=2$, one has $Y \hookrightarrow L^{10}(D)$. 
Therefore, the mapping $y\mapsto \mathcal{N}(y)$ is continuously Fr\'echet differentiable from $Y$ to $Y^*$. It is easy to check the boundedness of the operator $\mathcal{N}_y$, therefore Assumption \ref{assumption_N_Ny}(i) is fulfilled. Now, consider that $y_{u}(\omega)\in L^q(\Omega, H^1(D))$ for any $q$ arbitrary large. Let $K$ denote the embedding constant of $L^2(D)$ into $H^{-1}(D)$ and $\hat{K}$ denote the embedding constant of $L^{10}(D)$ into $H^1(D)$. For Assumption \ref{assumption_N_Ny}(ii), note that the continuity of the operator $y\mapsto \mathcal{N}(y)$ from $L^{q}(\Omega, H^1(D))$ to $L^{s}(\Omega,H^{-1}(D))$ for $q$ sufficiently large and $s$ sufficiently small comes from the following growth condition 
\begin{equation*}
    \| \mathcal{N}(y)\|_{H^{-1}(D)}\leq K \| y^5\|_{L^2(D)}\leq K \| y\|_{L^{10}(D)}^5\leq K \hat{K}\| y\|_{H^1(D)}^5 
\end{equation*}
for all $y\in H^1(D)$. 
A similar calculation along with \cite[Theorem 5]{GKT} yields the continuity of the operator $y\mapsto \mathcal{N}_y(y)$ from $L^{q}(\Omega, Y)$ to $L^{qs/(s-q)}(\Omega, \mathcal{L}(Y,Y^*))$. Similar arguments can be made for $\mathcal{A}$.

We verify now Assumption \ref{objective-function}(i). Note that the mapping $y\mapsto \tilde{J}(y)$ is continuous, and thus, Carath\'eodory; clearly, its Fr\'echet derivative is continuous as a mapping from $Y$ to $Y^*.$  
The Nemytskij operator $F'(y)= \tilde{J}_y (y(\cdot))$ is continuous from $L^q(\Omega,H^1(D))$ to $L^{\tilde{p}}(\Omega,H^{-1}(D))$ for $\tilde{p}=\tfrac{pq}{q-p}$ and any $1\leq p< q<\infty$. This can be verified by \cite[Theorem 4]{GKT} using the growth condition in Remark~\ref{rem-growth-condition}.
Thus, by \cite[Theorem 7]{GKT}, $F$ is continuously Fr\'echet differentiable from $L^q(\Omega, H^1(D))$ to $\mathcal{L}(L^q(\Omega, H^1(D)), L^p(\Omega))$. 

Now, we will verify the additional assumptions that were required in the previous section. Thanks to \eqref{eq:estimate_dual_arc_1}, the assumptions in Proposition \ref{prop:differentiability_j} hold true with $\alpha=1$ and the constants $c_1$, $\tilde{c}_1$, $c_2$, and $\tilde{c}_2$ belong to $L^{\infty}(\Omega)$; moreover, since $\rho(u) =\tfrac{\lambda}{2} \|u\|^2_{L^2(D)}$, we obtain the boundedness condition on the corresponding derivative.
Now, we show the additional assumptions from Lemma \ref{lemma:basic_estimates}. One can deduce by routine calculations that the Lipschitz condition for the adjoint mapping $\mathcal{N}_y^*$, 
as required in Lemma \ref{lemma:basic_estimates}, is satisfied by this model problem. Indeed, we observe that here
\begin{equation*}
   \| \mathcal{N}_y^*(y_2)-\mathcal{N}_y^*(y_1)\|_{\mathcal{L}(Y,Y^*)}= \sup_{\| h\|_{Y} =  \| v\|_{Y} =1} | \langle (\mathcal{N}_y^*(y_2)-\mathcal{N}_y^*(y_1))h,v \rangle_{Y^*,Y}|.
\end{equation*}
Moreover, for any $y_{i}\in Y=H^1(D)$ with $\| y_{i} \|_{L^{\infty}(D)}\leq M$ ($i=1,2)$,
\begin{align*}
&\left| \langle ( \mathcal{N}_y^* (y_2)- \mathcal{N}_y^* (y_1) )  h,v \rangle_{Y^*,Y} \right|\leq  \int_D 5\left| (y_2(x)^4-y_1(x)^4) h(x) v (x)\right| \D x  \\
    &= \int_D 5 \left| (y_2(x)+y_1(x))(y_2(x)^2+ y_1(x)^2)(y_2(x)-y_1(x)) h(x)v(x)\right|
    \D x  \\
    &\leq  20 M^3 \int_D \left| (y_2(x)-y_1(x)) h(x)v(x)\right|  \D x \\
    &\leq 20 M^3 \| y_2-y_1 \|_{L^{p_1}(D)} \| h \|_{L^{p_2}(D)}\| v \|_{L^{p_3}(D)},
\end{align*}
where in the last step we have utilized the H\"older inequality with $p_1,p_2,p_3 \geq 2$ such that $\frac{1}{p_1}+\frac{1}{p_2}+\frac{1}{p_3}=1$ (see, e.g., \cite[Lemma 1.13]{Hinze}). By the Sobolev Embedding Theorem  (cf.~\cite[Theorem 5.4]{a}), since $D\subset \mathbb{R}^2$ and $H^1(D) \hookrightarrow L^q(D)$  for all $2 \leq q < \infty$, we have that
\begin{align*}
 |  \langle \left( \mathcal{N}_y (y_2)- \mathcal{N}_y (y_1)\right)  h,v \rangle_{Y^*,Y}| \leq 
20 M^3 K^3 \| y_2-y_1 \|_{H^{1}(D)} \| h \|_{H^{1}(D)}\| v \|_{H^{1}(D)},
\end{align*}
where $K$ is the embedding constant. We see that for any $y_{i}\in Y=H^1(D)$ with $\| y_{i} \|_{L^{\infty}(D)}\leq M$ ($i=1,2)$,
\begin{equation}\label{eq:Lipsch_N_ex}
    \| \mathcal{N}_y^*(y_2)-\mathcal{N}_y^*(y_1)\|_{\mathcal{L}(Y,Y^*)} \leq   20 M^3 K^3 \| y_2-y_1 \|_{H^{1}(D)}
\end{equation}
and this proves the Lipschitz condition with for the mapping $\mathcal{N}_y^*(\cdot)$ as required. Since $y \mapsto D_y \tilde{J}(y,\omega)$ is locally Lipschitz, the assumptions of Lemma \ref{lemma:basic_estimates} are satisfied.

Now, we turn to Proposition \ref{prop:2.2verified}. By \cite[Theorem 4.7]{tro}, 
we have that condition \eqref{eq:casas_L_infty_estim} must hold true in this example. By replicating the proof of the aforementioned result, one can infer information on the integrability of $c_{\infty}$ which is, in principle, a random variable for PDEs with random data. 
Notice that when $c_{\infty}$ is independent of $\omega$, then the constant $L(M,\omega)$ appearing in Proposition \ref{prop:2.2verified} is in fact independent of $\omega$, as are the remaining terms in the constant $c_3.$ Hence $c_3 \in L^1(\Omega)$ is certainly fulfilled for this example. The complete analysis of the integrability of $c_{\infty}$ is beyond the scope of this paper and is left as a topic for future research.

\paragraph{Numerical details}
Simulations were run using FEniCS \cite{Alnes2015}. 
A uniform mesh with 1250 shape regular triangles was used.
The control, state, and adjoint were discretized using piecewise linear finite elements. The initial control was chosen to be $u_1(x) \equiv 1.$ 

The step size is chosen to be $t_n = \theta/n$ with $\theta = \tfrac{2}{\lambda}$, which is informed by the rule for the strongly convex case; cf.~\cite{GP}. 

\paragraph{Results}
To verify the convergence rate 
\eqref{eq:conv_est_g} in Theorem~\ref{thm:convergence-rate}, we use a sample average approximation (SAA) of the problem with $N=5,000$ randomly drawn samples $\xi_j=(\xi_j^1, \dots, \xi_j^{20})$. That is, we solve the approximate problem
\begin{equation*}
    \begin{aligned}
    &\min_{u \in L^2(D)} \, \left\lbrace \hat{j}_N(u):=\frac{1}{2N} \sum_{j=1}^N \| \hat{y}(\cdot, \xi_j)-y_D \|_{L^2(D)}^2 + \frac{\lambda}{2}\| u \|_{L^2(D)}^2 \right\rbrace\\
    & \qquad \text{s.t.} \quad \hat{y}(\cdot, \xi_j) \text{ solves \eqref{eq:ex-PDE} with } a=\hat{a}(\cdot,\xi_j).
    \end{aligned}
\end{equation*}

 Algorithm~\ref{alg:PSG_Hilbert_Nonconvex_Unconstrained} was run once each for three different values of the regularization parameter $\lambda$ with a single sample (taken from the SAA set) per iteration. At each iterate, the full gradient was computed; note that this is done for verifying convergence rates, only: Algorithm~\ref{alg:PSG_Hilbert_Nonconvex_Unconstrained} does not require more than one sample per iteration.  The test is terminated after 300 iterations (or if $\min_{t=1, \dots, n}\lVert \nabla \hat{j}_N(u_t) \rVert^2 \leq 10^{-8}$ for the deterministic experiment). The sequence of random vectors was generated using the \texttt{numpy} seed \texttt{numpy.random.seed(10)}. The 5,000 random vectors $\xi_j=(\xi_j^1, \dots, \xi_j^{20})$ were generated one after the other; i.e., we used the seed to generate the sequence $(\xi_1^1, \dots, \xi_1^{20}, \xi_2^1, \dots, \xi_2^{20}, \dots, \xi_{5000}^{20}).$ After generating the vectors and using the same seed, random indices are chosen from the set $\{1, 2, \dots, 5000 \}$ for 300 iterations for each value of $\lambda$ (in the order $\lambda=1$, then $\lambda=0.1$, then $\lambda=0.01$). 

The corresponding control and state obtained at the final iterate are displayed in Figures~\ref{fig:solution-1}--\ref{fig:solution-3} along with a reference curve for the convergence rate dictated by Theorem~\ref{thm:convergence-rate}. We see that the step-size yields good performance of the algorithm for this example. Note that the final state depends on the sample drawn at the last iterate. For the sake of comparison, the deterministic solution for the case $\lambda=0.01$ is displayed in Figure~\ref{fig:solution-4}. Comparing that solution with Figure~\ref{fig:solution-3}, one sees the modest impact of uncertainty for this model.  


\begin{figure}
    \centering
        \includegraphics[width=6cm]{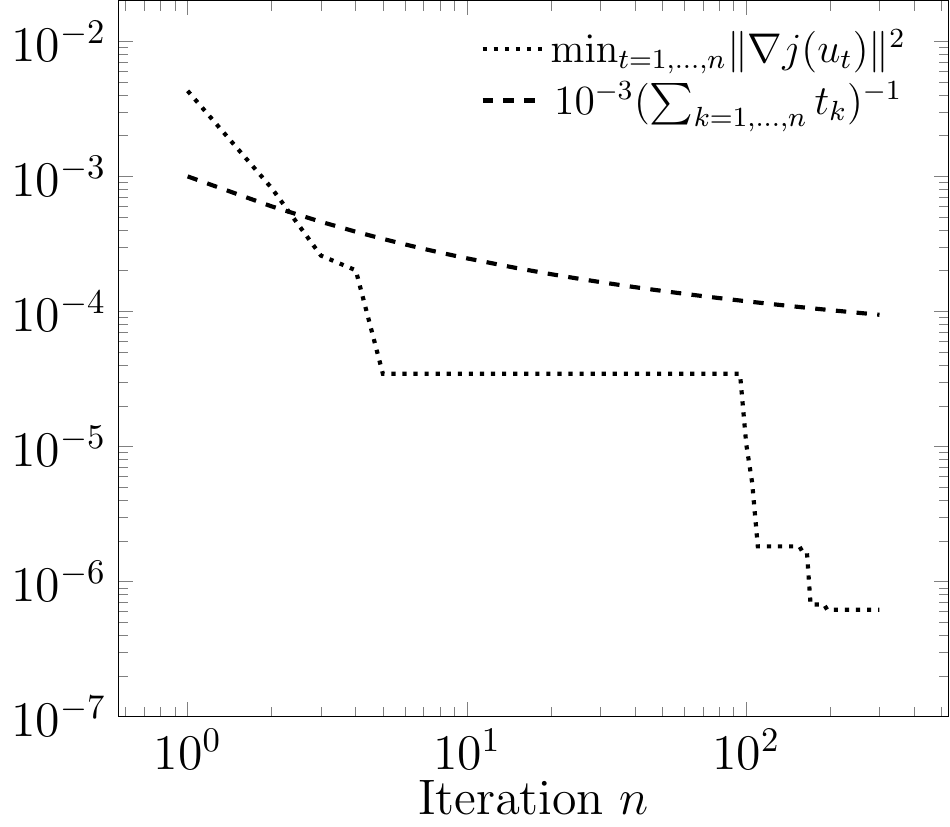}\hspace{1cm}
            \includegraphics[width=6cm]{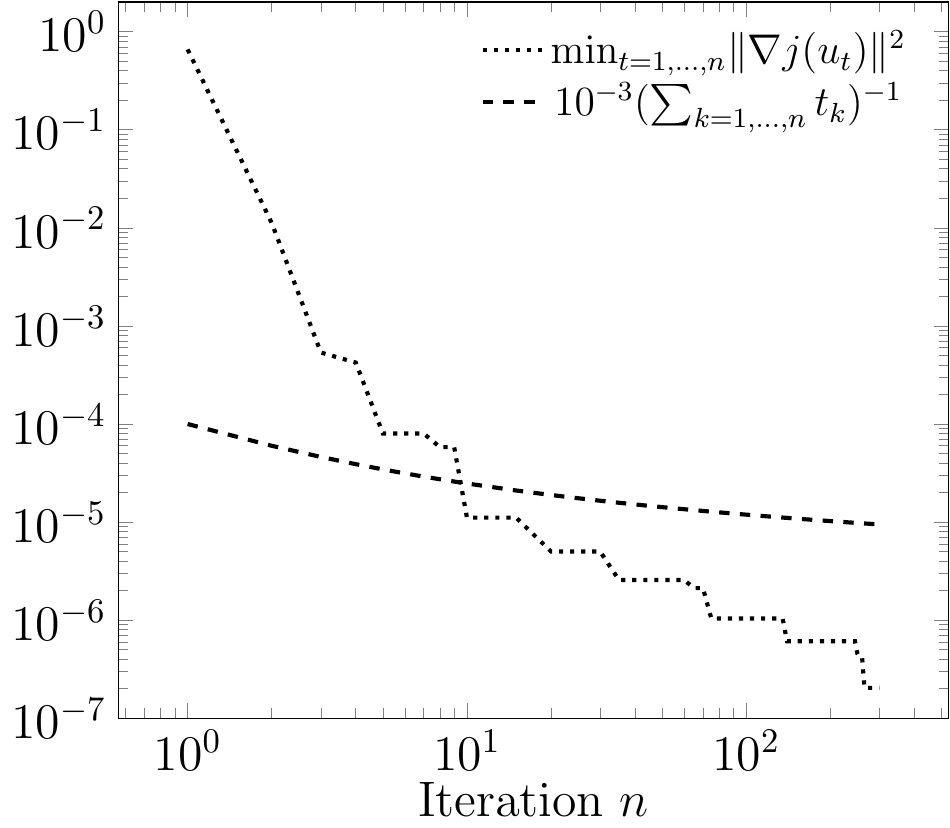}
    \includegraphics[width=6cm]{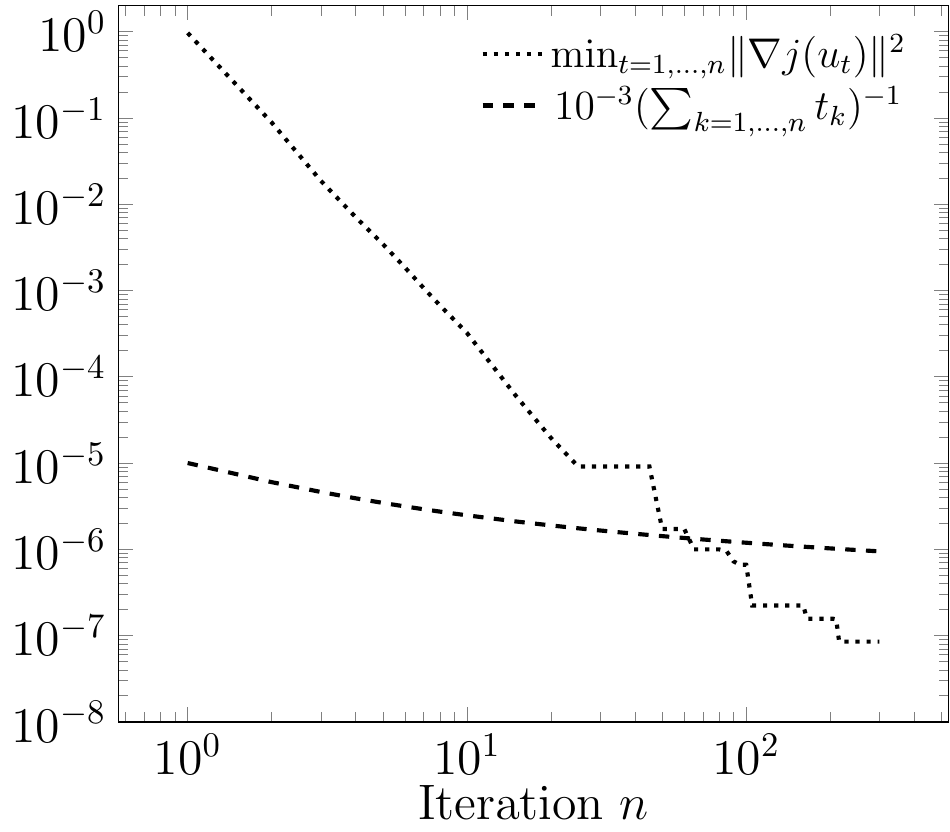}
    \caption{Convergence for $\lambda=1$ (top left), $\lambda=0.1$ (top right), and $\lambda=0.01$ (bottom)}
    \label{fig:convergence}
\end{figure}

\begin{figure}
    \centering
        \includegraphics[width=5.8cm]{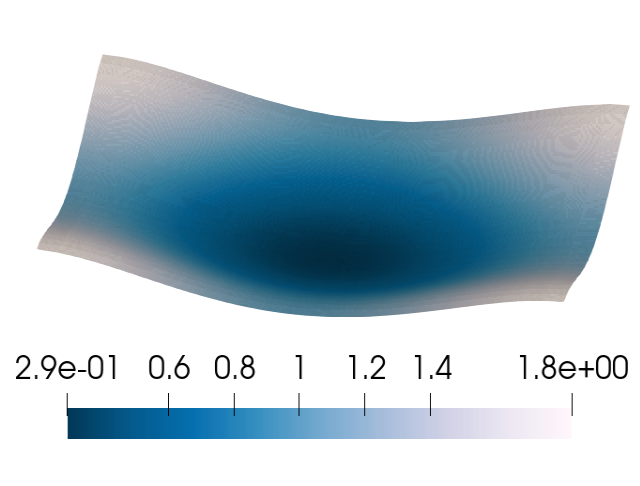} \hspace{0.5cm}
            \includegraphics[width=5.8cm]{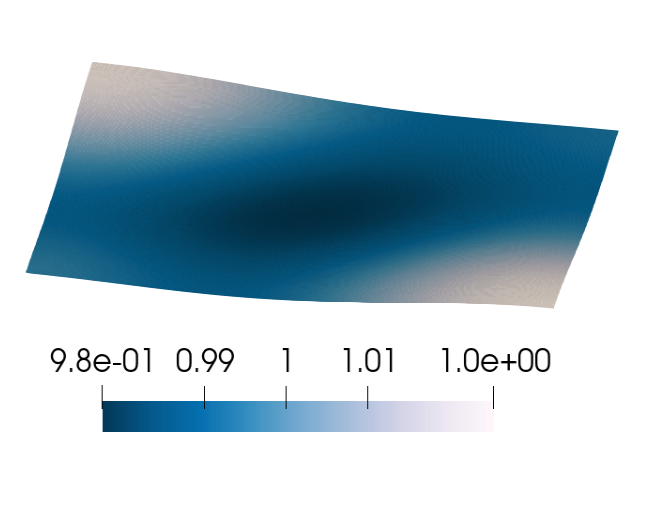}
    \caption{Control $u^*$ and state $y^*$ for $\lambda=1$}
    \label{fig:solution-1}
\end{figure}

\begin{figure}
    \centering
        \includegraphics[width=5.8cm]{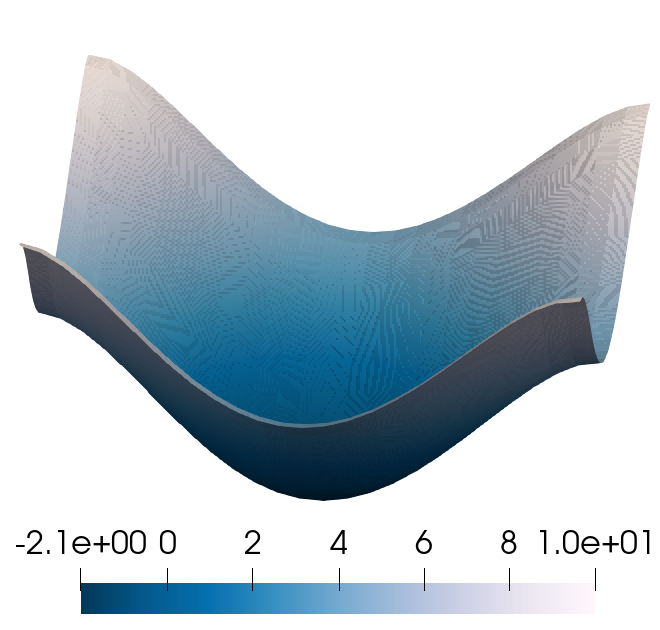} \hspace{0.5cm}
            \includegraphics[width=5.8cm]{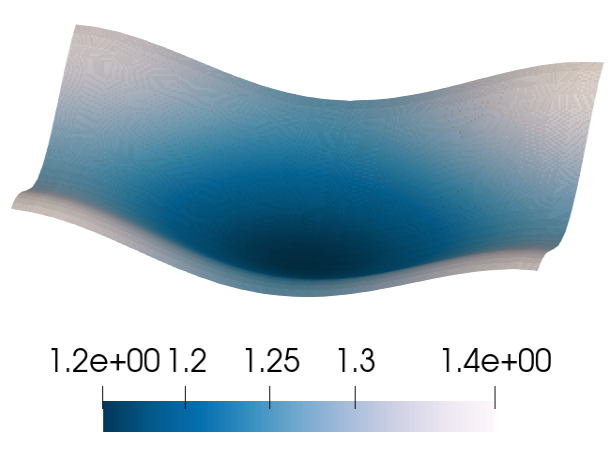}
    \caption{Control $u^*$ and state $y^*$ for $\lambda=0.1$}
    \label{fig:solution-2}
\end{figure}

\begin{figure}
    \centering
        \includegraphics[width=5.8cm]{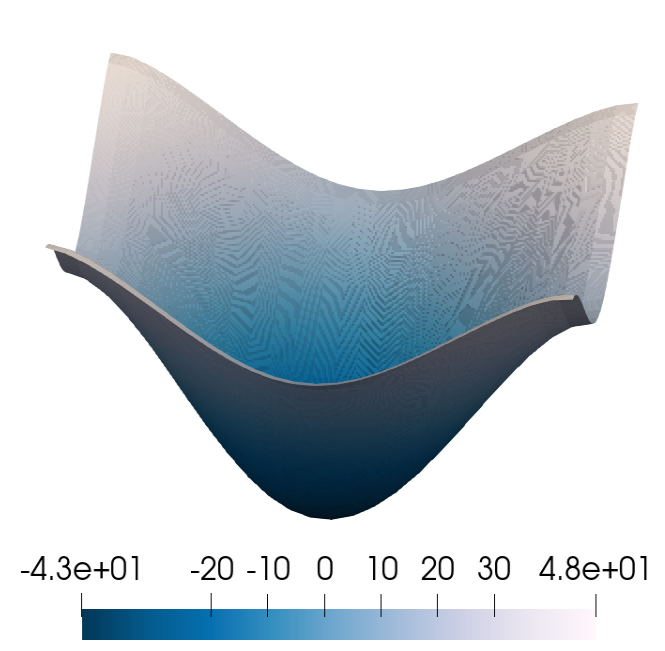}
        \hspace{0.5cm}
            \includegraphics[width=5.8cm]{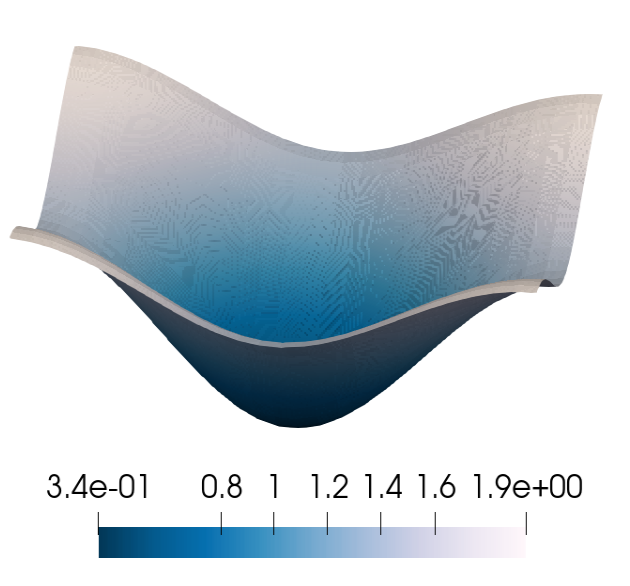}
    \caption{Control $u^*$ and state $y^*$ for $\lambda=0.01$}
    \label{fig:solution-3}
\end{figure}

\begin{figure}
    \centering
        \includegraphics[width=5.8cm]{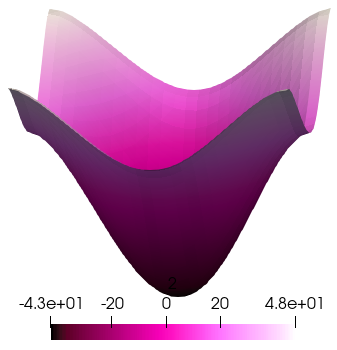}
        \hspace{0.5cm}
            \includegraphics[width=5.8cm]{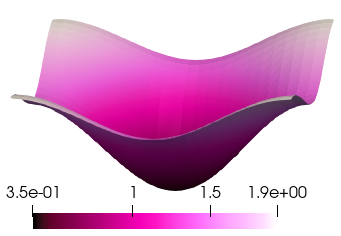}
    \caption{\textit{Deterministic} control $u^*$ and state $y^*$ for $\lambda=0.01$ and $a \equiv 1$}
    \label{fig:solution-4}
\end{figure}

\section{Conclusion}
In this paper, we continued our initial investigation \cite{GS2021} of nonconvex stochastic optimization problems in Hilbert space. We filled in missing aspects not treated in \cite{GS2021}, which focused on asymptotic convergence results for nonsmooth problems. The techniques used there relied on the more involved ODE method; in the smooth case, an asymptotic convergence result is attainable using standard arguments, as we demonstrated here. We further provided theoretical results concerning weak convergence and convergence rates of the stationarity measure. We were able to demonstrate that, for a large class of optimal control problems subject to semilinear PDEs under uncertainty, the assumptions we require for the stochastic gradient method to converge are reasonable, including measurability, which was not handled in previous work. Numerical experiments demonstrated the expected convergence rates.

\section*{Acknowledgements}
We would like to thank the reviewers for their careful reading. The second author acknowledges the support of the GNAMPA project ``Problemi inversi e di controllo per equazioni di evoluzione e loro applicazioni'' funded by INdAM and coordinated by Prof.~Floridia from the University of Reggio Calabria (Italy). The main part of this research was written while the second author was working at the Department of Information Engineering, Computer Science and Mathematics of the University of L'Aquila (Italy).
\section*{Appendix}
Here we list the assumptions and some basic results that are used in Section \ref{sec:OCforPDEs} when we study the PDE-constrained optimization problem. The following four assumptions are copied from \cite{KS}. Below, pointwise statements with respect to $\omega$ are always assumed to hold almost surely.
\begin{assumption}\label{hp1}
The operator $\cA\colon \Omega \rightarrow\mathcal{L}(Y,Y^*)$ is such that $\cA(\omega)$ is monotone and there exists a constant $\gamma>0$ and a positive random variable $C\colon \Omega \rightarrow [0,\infty)$ such that 
\begin{equation}\label{coercivity}
\langle \cA(\omega)y,y\rangle_{Y^*,Y}\geq C(\omega)\| y \|_{Y}^{\gamma + 1} \quad \forall y \in Y. 
\end{equation}
In addition, $\cN\colon Y\times \Omega \rightarrow Y^*$ is such that $\cN(\cdot, \omega)$ is maximally monotone with $\cN(0,\omega)=0$, and $\cb$ is a given function $\cb\colon \Omega \rightarrow Y^*$. Moreover, the operator $\cB\colon \Omega \rightarrow \mathcal{L}(U,Y^*)$ is such that $\cB(\omega)$ is completely continuous.
\end{assumption}

 \begin{assumption}\label{assumption_diff}
Suppose that there exists $s,t\in [1,\infty]$ with 
 \begin{equation*}
   1+ \frac{1}{\gamma} \leq s < \infty , \quad t\geq \frac{s}{\gamma (s-1)-1},
 \end{equation*}
 such that $\cA(\cdot)y\in L^s(\Omega,Y^*)$ for all $y\in Y$, 
 $\cN(y,\cdot)\in L^s(\Omega, Y^*)$ for all $y \in Y$, $\cB\in L^s(\Omega,\mathcal{L}(U,Y^*))$, $\mathfrak{b}\in L^s(\Omega,Y^*)$, and $C^{-1}\in L^t(\Omega)$. 
 \end{assumption}
 
\begin{assumption}\label{assumption_N_Ny}
 \setcounter{subassumption}{0}
\subasu The mapping $y\mapsto \cN(y,\omega)$ is continuously Fr\'echet differentiable from $Y$ into $Y^*$. Moreover, the partial derivative $\cN_y(\cdot,\omega)$ defines a bounded (nonnegative~\footnote{Note that monotonicity of $y \mapsto \mathcal{N}(y,\omega)$ implies nonnegativity.} linear) operator from $Y$ to $Y^*$.\label{assumption_N_Ny-1} \\
\subasu The maps $\cA$ and $y \mapsto \cN(y,\cdot)$ are continuous from $L^q(\Omega,Y)$ into $L^s(\Omega, Y^*)$ and $y\mapsto \cN_y(y,\cdot)$ is a continuous map from $L^q(\Omega,Y)$ into $L^{\frac{qs}{s-q}}(\Omega,\mathcal{L}(Y,Y^*)$).
 \end{assumption}
 

\begin{assumption}\label{objective-function}
 \setcounter{subassumption}{0}
\subasu The function $\tilde{J}\colon Y\times \Omega \rightarrow \R$ is a Carath\'eodory function with $\tilde{J}(\cdot,\omega)$ is continuously Fr\'echet differentiable with respect to $y\in Y$. The superposition $F$ defined by $F(y):=\tilde{J}(y(\cdot),\cdot)$ is continuously Fr\'echet differentiable from  $L^q(\Omega,Y)$ into $ L^p(\Omega)$ with derivative $F'(y)= \tilde{J}_y(y(\cdot),\cdot)\in L^{\frac{pq}{q-p}}(\Omega, Y^*)$.\\
\subasu $\rho\colon U\rightarrow \mathbb{R}$ is convex and continuously Fr\'echet differentiable.
\end{assumption}

For our risk-neutral setting, $p=1$ is sufficient. Finally, to verify Fr\'echet differentiability of the objective $j(u)=\E[J(u,\cdot)]$, we use the following lemma from \cite[Lemma C.3]{GS2021}. We use the notation $J_\omega:=J(\cdot,\omega)$.
\begin{lemma}
\label{lemma:frechet-exchange-derivative-expectation}
Suppose $V\subset X$ is a open neighborhood of a Banach space $X$ containing $u$, and (i) the expectation $j(v)$ is well-defined and finite-valued for all $v \in V$, (ii) for almost every $\omega \in \Omega$, the functional $J_{\omega}\colon X \rightarrow \R$ is Fr\'echet differentiable at $u$, (iii) there exists a positive random variable $C_J \in L^1(\Omega)$ such that for all $v \in V$ and almost every $\omega \in \Omega$,
 \begin{equation*}
  \label{eq:randomLipschitz}
\lVert J'_{\omega}(v) \rVert_{X^*} \leq C_J(\omega).
 \end{equation*}  Then $j$ is Fr\'echet differentiable at $u$ and $j'(u) = \E[J'_{\omega}(u)].$
\end{lemma}




\end{document}